\documentclass[final,dvips,usenames]{siamltex}

\usepackage{color}
\usepackage{amssymb,amsmath,algorithm2e,epsfig}
\usepackage{bm}

\usepackage{tikz}
\usetikzlibrary{arrows}
\usetikzlibrary{decorations.markings}

\newtheorem{remark}[theorem]{Remark}

\newcommand{\Aslv}{A$_\mathtt{S}$}
\newcommand{\Aest}{A$_\mathtt{E}$}
\newcommand{\Amrk}{A$_\mathtt{M}$}
\newcommand{\Aref}{A$_\mathtt{R}$}
\renewcommand{\b}[1]{\boldsymbol{#1}}

\newcommand{\cA}{\mathcal{A}}

\newcommand{\cF}{\mathcal{F}}

\newcommand{\CFD}{C_\mathrm{F,D}}

\newcommand{\cM}{\mathcal{M}}
\newcommand{\CP}{C_\mathrm{P}}

\newcommand{\cT}{\mathcal{T}}

\newcommand{\cTstar}{\cT_*}

\newcommand{\ddiv}{\operatorname{div}}
\newcommand{\dx}[1][x]{\,\mathrm{d}#1}
\newcommand{\ellfin}{{\ell_\mathrm{fin}}}
\newcommand{\GammaD}{{\Gamma_{\mathrm{D}}}}
\newcommand{\GammaN}{{\Gamma_{\mathrm{N}}}}
\newcommand{\Hdiv}[1][\Omega]{\boldsymbol{H}(\ddiv,#1)}
\newcommand{\HdivO}[1][\Omega]{\boldsymbol{H}_0(\ddiv,#1)}

\newcommand{\interior}{\operatorname{int}}
\newcommand{\kfin}{{k_\mathrm{fin}}}

\newcommand{\norm}[1]{\|#1\|}

\newcommand{\oGammaD}{\overline\Gamma_{\mathrm{D}}}
\newcommand{\oGammaN}{\overline\Gamma_{\mathrm{N}}}

\newcommand{\R}{\mathbb{R}}

\newcommand{\Erel}{E^\mathrm{rel}} 
\newcommand{\Erelest}{E^\mathrm{rel}_\mathrm{est}} 
\newcommand{\Ereltol}{E^\mathrm{rel}_\mathrm{TOL}} 

\newcommand\SOLVE{\mbox{\texttt{SOLVE}}}
\newcommand\ESTIMATE{\mbox{\texttt{ESTIMATE}}}

\newcommand\MARK{\mbox{\texttt{MARK}}}
\newcommand\REFINE{\mbox{\texttt{REFINE}}}

\newcommand\Ta{\cT_{\boldsymbol{a}}}
\newcommand\poma{\partial\Oma} 

\newcommand\cETK{\mathcal{E}^\cT_K}
\newcommand\cET{\mathcal{E}^\cT}
\newcommand\cETD{\mathcal{E}^\cT_{\mathrm{D}}}
\newcommand\cETI{\mathcal{E}^\cT_{\mathrm{I}}}
\newcommand\cETN{\mathcal{E}^\cT_{\mathrm{N}}}
\newcommand\cEaB{\mathcal{E}^{\mathrm{B}}_{\boldsymbol{a}}}
\newcommand\cEaI{\mathcal{E}^{\mathrm{I}}_{\boldsymbol{a}}}
\newcommand\cEaBN{\mathcal{E}^{\mathrm{B,N}}_{\boldsymbol{a}}}
\newcommand\cEaBD{\mathcal{E}^{\mathrm{B,D}}_{\boldsymbol{a}}}
\newcommand\cEaBE{\mathcal{E}^{\mathrm{B,E}}_{\boldsymbol{a}}}

\newcommand\pt{\partial}

\newcommand\RT{\mathbf{RT}_p}
\newcommand\bW{\mathbf{W}}
\newcommand\bWa{\mathbf{W}_\ta}
\newcommand\bWT{\mathbf{W}^\cT}

\newcommand\PiWK{{\bm \Pi}^{\cA^{-1}}_K}
\newcommand\PiWTa{{\bm \Pi}^{\cA^{-1}}_{\Ta}}

\newcommand\PpTast{P_{p}^{*}(\Ta)}

\newcommand\PpT{P_p(\Ta)}

\newcommand{\ttau}{\bm \tau}
\newcommand{\ttauT}{\bm \tau^{\cT}}
\newcommand{\varphiT}{\varphi^{\cT}}

\newcommand\VT{V^{\cT}}

\newcommand\VTstar{V^{\cTstar}}

\newcommand\NT{\mathcal{N}^{\cT}}
\newcommand\NTI{\mathcal{N}^{\cT}_{\mathrm{I}}}
\newcommand\NTN{\mathcal{N}^{\cT}_{\mathrm{N}}}
\newcommand\NTD{\mathcal{N}^{\cT}_{\mathrm{D}}}
\newcommand\NTK{\mathcal{N}^{\cT}_{K}}

\newcommand\NTGa{\mathcal{N}^{\cT}_{\Gamma}}

\newcommand\omK{\omega_K}

\newcommand\tomK{{\tilde\omega}_K}

\newcommand\bR{\mathcal{R}}

\newcommand{\HF}{\psi_{\boldsymbol{a}}}

\newcommand\tO{{\bm 0}}

\newcommand\bw{{\boldsymbol w}}
\newcommand\bwT{{\boldsymbol w}^{\cT}}
\newcommand\tn{{\boldsymbol n}}
\newcommand\tq{{\boldsymbol q}}
\newcommand\ta{{\boldsymbol a}}

\newcommand\tqT{{\boldsymbol q}^{\cT}}

\newcommand\lamTf{\lambda^{\cT}_i}
\newcommand\lamTkf{\lambda^{\cT_k}_i}

\newcommand\lamT{\lambda^{\cT}}

\newcommand\uTf{u^{\cT}_i}
\newcommand\uTkf{u^{\cT_k}_i}
\newcommand\uT{u^{\cT}}

\newcommand\vT{v^{\cT}}
\newcommand\tqa{{\boldsymbol q}^{\cT}_{\boldsymbol{a}}}
\newcommand\ttqa{{\boldsymbol {\tilde{q}}}^{\cT}_{\boldsymbol{a}}}
\newcommand\tqaO{{\boldsymbol q}^{\cT, 0}_{\boldsymbol{a}}}
\newcommand\da{d^{\cT}_{\boldsymbol{a}}}
\newcommand\daO{d^{\cT, 0}_{\boldsymbol{a}}}
\newcommand\tsa{{\boldsymbol s}^{\cT}_{\boldsymbol{a}}}
\newcommand\gTa{g^\cT_\ta}

\newcommand\Oma{\omega_{\ta}}

\newcommand\rTa{r^{\cT}_{\boldsymbol{a}}}

\newcommand{\bd}{\begin{definition}}
\newcommand{\ed}{\end{definition}}
\newcommand{\be}{\begin{equation}}
\newcommand{\ee}{\end{equation}}
\newcommand{\bse}{\begin{subequations}}
\newcommand{\ese}{\end{subequations}}
\newcommand{\br}{\begin{remark}}
\newcommand{\er}{\end{remark}}
\newcommand{\bc}{\begin{corollary}}
\newcommand{\ec}{\end{corollary}}

\title{Two-sided bounds of eigenvalues -- local efficiency and convergence of adaptive algorithm%
\thanks{The support of I.\ \v{S}ebestov\'a by Fondecyt Postdoctoral Grant no.~3150047 and the support of T.~Vejchodsk\'y by the project no.~P101/14-02067S of the Czech Science Foundation and by RVO~67985840
are gratefully acknowledged.}}

\author{Ivana \v{S}ebestov\'a\footnotemark[2]
        \and Tom\'a\v s Vejchodsk\'y\footnotemark[3]}

\begin{document}

\maketitle

\renewcommand{\thefootnote}{\fnsymbol{footnote}}

\footnotetext[2]{Departamento de Ingenier\'\i a Matem\' atica, Facultad de Ciencias F\'\i sicas y Matem\'aticas,
Universidad de Concepci\' on, Casilla 160-C, Concepci\' on,
Chile, email: ({\tt isebestova@udec.cl}).}
\footnotetext[3]{Institute of Mathematics, Czech Academy of Sciences, {\v Z}itn{\'a} 25, CZ-115 67 Praha 1, Czech Republic, ({\tt vejchod@math.cas.cz}).}

\renewcommand{\thefootnote}{\arabic{footnote}}

\begin{abstract}
We generalize and analyse the method for computing lower bounds of the principal eigenvalue 
proposed in 
our previous paper (I.~\v{S}ebestov\'a, T.~Vejchodsk\'y, SIAM J. Numer. Anal. 2014).
This method is suitable for symmetric elliptic eigenvalue problems with mixed boundary conditions 
of Dirichlet, Neumann, and Robin type and
it is based on a posteriori error analysis using flux reconstructions.
We improve the original result in several aspects.
We show how to obtain lower bounds even for higher eigenvalues.
We present a local approach for the flux reconstruction enabling efficient implementation.
We prove the equivalence of the resulting estimator with the classical residual estimator
and consequently its local efficiency. 
We also prove the convergence of the corresponding adaptive algorithm.
Finally, we illustrate the practical performance of the method by numerical examples.
\end{abstract}



\pagestyle{myheadings}
\thispagestyle{plain}
\markboth{IVANA \v{S}EBESTOV\'A  AND TOM\'A\v{S} VEJCHODSK\'Y}{TWO-SIDED BOUNDS OF EIGENVALUES}


\section{Introduction}

Galerkin method provides a simple and efficient way how to compute approximate eigenvalues and eigenfunctions of differential operators. For symmetric elliptic problems this method naturally yields upper bounds on the exact eigenvalues. Computation of an accurate lower bound is a much more complicated task. It is an old problem and many authors have already approached it from different perspectives.

Concerning recent results, various nonconforming methods to compute the lower bounds \cite{AndRac:2012,HuHuaLin2014,LinXieLuoLiYan:2010,LuoLinXie:2012,Rannacher:1979,YanZhaLin:2010} have been proposed. These approaches provide typically an asymptotic lower bound in the sense that the lower bound is guaranteed only if the corresponding discretization mesh is sufficiently fine. Guaranteed lower bound for the Laplace eigenvalues with homogeneous Dirichlet boundary conditions are obtained even on coarse meshes in \cite{CarGed2014} by using Crouzeix-Raviart nonconforming finite elements. 
A generalization of this approach to a biharmonic operator is provided in \cite{CarGal2014}.
A lower bound on the smallest eigenvalue is obtained in \cite{Repin2012} by a nonoverlapping decomposition of the domain into subdomains, where the exact eigenvalues are known. 
In a sense similar method is proposed in \cite{Kuz_Rep_Guar_low_bound_smal_eig_elip_13}. It is based on an overlapping decomposition of the domain into geometrically simple subdomains and it yields a lower bound on the smallest eigenvalue for homogeneous Neumann or mixed Neumann-Dirichlet boundary conditions. A lower bound on the smallest eigenvalue for a triangle is obtained in \cite{Kobayashi2015} using a scaling. An interesting generalization of the method of eigenvalue inclusions \cite{BehnkeGoerish1994,Plum1997} for the Maxwell operator is provided in \cite{Barrenechea2014}.

In \cite{Seb_Vejch_2sidedb_eigen_Fr_Poin_trace_14} we propose another approach based on a combination of the method of a priori-a posteriori inequalities \cite{Sigillito:1977,KutSig:1978} and a complementarity technique \cite{Complement:2010,systemaee:2010}. The main result of \cite{Seb_Vejch_2sidedb_eigen_Fr_Poin_trace_14} is the description of the method and a proof that it yields a lower bound on the principal eigenvalue. However, the crucial flux reconstruction is obtained there by solving a straightforward but global minimization problem. Here, we improve on this by considering the error estimator with the flux reconstructed locally. This enables an efficient and naturally parallel implementation. We use the local flux reconstruction originally proposed in \cite{BraSch:2008} for source problems and we modify it for eigenvalue problems.

The heart of this paper is the proof of the equivalence of the estimator based on the local flux reconstruction with the classical residual estimator. This result has important consequences such as the local efficiency of the proposed estimator and the convergence of the corresponding adaptive algorithm, which we also prove. In addition, we show how to compute lower bounds for theoretically arbitrary eigenvalue. This is an improvement over \cite{Seb_Vejch_2sidedb_eigen_Fr_Poin_trace_14}, where lower bounds for the principal (smallest) eigenvalue only are considered.

In Section~\ref{se:prob_def}, we define the eigenvalue problem in an abstract way using a pair of symmetric bilinear forms on a Hilbert space. We briefly summarize the key abstract results, emphasize the importance of compactness, and provide a lemma that translates the classical compactness results, such as the Rellich and trace theorems to the required compactness of the solution operator. We also recall an abstract theorem from \cite{Seb_Vejch_2sidedb_eigen_Fr_Poin_trace_14} yielding lower bounds on the principal eigenvalues.
Section~\ref{se:problem} introduces a symmetric elliptic eigenvalue problem with mixed Dirichlet, Neumann, and Robin boundary conditions and its finite element discretization.
Section~\ref{se:classicalest} recalls the classical residual error indicators and their local efficiency with respect to the residual.
Section~\ref{se:errest} defines the error estimator based on the local flux reconstruction.
In Section~\ref{se:loceff} we prove that this estimator is controlled by the classical residual estimator and hence that it is locally efficient as well.
Section~\ref{se:adaptalg} reviews the general assumptions for the convergence of the adaptive algorithm and shows that the adaptive algorithm driven by the proposed error indicators converges. 
Section~\ref{se:numex} presents numerical results and illustrates the practical performance of the method.
Finally, Section~\ref{se:concl} draws conclusions.

\section{Abstract setting}
\label{se:prob_def}

This section briefly describes the general setting of eigenvalue problems based on a pair of bilinear forms in
a Hilbert space. This general setting enables to treat the standard types of eigenvalue problems such as the Dirichlet, Neumann, Steklov, etc. in a unified manner.

Let $V$ be a real Hilbert space with a scalar product $a(u,v)$ for $u,v\in V$. In particular the form $a(u,v)$ is continuous, bilinear, symmetric and positive definite. Further, let a form $b(u,v)$ for $u,v\in V$ be continuous, bilinear, symmetric, and positive semidefinite, i.e. $b(v,v) \geq 0$ for all $v \in V$. We use notation $\norm{v}_a^2 = a(v,v)$ and $|v|_b^2 = b(v,v)$ for the norm induced by the scalar product $a$ and the seminorm induced by the bilinear form $b$, respectively.
We will consider an abstract eigenvalue problem to find an eigenvalue $\lambda_i\in\R$ and a nonzero eigenfunction $u_i \in V$ such that
\begin{equation}
  \label{eq:EP1}
  a(u_i,v) = \lambda_i b(u_i, v) \quad \forall v \in V.
\end{equation}
The positivity of eigenvalues $\lambda_i$ and the positivity of the seminorm $|u_i|_b$ of the corresponding eigenfunctions is easy to show.
\begin{lemma}
\label{le:lambdapos}
Let $u_i \in V$ be an eigenfunction of \eqref{eq:EP1} corresponding to an eigenvalue $\lambda_i\in\R$, then $|u_i|_b > 0$ and $\lambda_i > 0$.
\end{lemma}
\begin{proof}
It follows immediately from \eqref{eq:EP1} and the facts that $u_i \neq 0$, $a(u_i,u_i) > 0$,
and $b(u_i,u_i) \geq 0$.
\end{proof}

In order to verify the well-posedness of eigenproblem \eqref{eq:EP1}, we consider the solution operator $S : V \rightarrow V$. Given $u\in V$, the element $Su \in V$ is defined by identity
\begin{equation}
\label{eq:solop}
  a(Su,v) = b(u,v) \quad \forall v \in V.
\end{equation}
The existence and uniqueness of $Su$ as well as the linearity and continuity of $S$ follow from the Riesz representation theorem.

Having defined $S$, we assume it is compact. This is a crucial assumption and the subsequent analysis relies on it. Further, from the symmetry of both $a$ and $b$ we easily obtain that the operator $S$ is selfadjoint in $V$, i.e. $a(Su,v) = a(u,Sv)$ for all $u,v\in V$.
These properties of $S$ enable to use the Hilbert--Schmidt spectral theorem \cite[Theorem~4, Chapter~II, section~3]{Gaal_Lin_anal_repres_theo_73}. Considering the eigenproblem
\begin{equation}
\label{eq:EP2}
  S u_i = \mu_i u_i,
\end{equation}
the Hilbert--Schmidt spectral theorem implies the existence of
a countable sequence $\{u_i\}$ of eigenfunctions corresponding to nonzero eigenvalues $\mu_i$.
These eigenfunctions are orthogonal, i.e. $a(u_i,u_j) = 0$ for all $i\neq j$, and they generate
a subspace $\mathcal{M}$ with the property
\begin{equation}
  \label{eq:Vsplit}
  V = \mathcal{M} \oplus \operatorname{ker}(S).
\end{equation}
Here, $\oplus$ denotes the direct sum and $\operatorname{ker}(S) = \{ v \in V : Sv = 0 \}$ is the kernel of $S$. In what follows, we will consider eigenfunctions $u_i$ to be normalized as
\begin{equation}
\label{eq:normalization}
  b(u_i,u_j) = \delta_{ij}, \quad \forall i,j=1,2,\dots,
\end{equation}
where we use the Kronecker delta.
Note that this normalization is well defined due to Lemma~\ref{le:lambdapos}.

It is not surprising that eigenproblems \eqref{eq:EP1} and \eqref{eq:EP2} are linked and we can derive properties of eigenproblem \eqref{eq:EP1} from \eqref{eq:EP2}.

\begin{lemma}
\label{le:Sgamma}
Under the above setting, the following statements hold true.
\begin{enumerate}
\item[1.]
Number $\lambda_i \in \R$ is an eigenvalue corresponding
to the eigenfunction $u_i \in V$ of \eqref{eq:EP1}
if and only if
$\mu_i = 1/\lambda_i$ is a nonzero eigenvalue corresponding
to the eigenfunction $u_i$ of the operator $S$; see \eqref{eq:EP2}.
\item[2.]
The number of eigenvalues $\lambda_i$ of \eqref{eq:EP1} such that $\lambda_i \leq M$
is finite for any $M>0$.
\item[3.]
The smallest eigenvalue of \eqref{eq:EP1} is given by
$\lambda_1 = \inf\limits_{u\in V, |u|_b\neq 0}
\norm{u}_a^2/|u|_b^2
$.
\end{enumerate}
\end{lemma}
\begin{proof}
1.\ Using \eqref{eq:solop} in \eqref{eq:EP1}, we obtain
identity $a(u_i, v) = \lambda_i a( S u_i, v)$
for all $v \in V$. This is clearly equivalent
to \eqref{eq:EP2} with $\mu_i = 1/\lambda_i$ provided that $\lambda_i \neq 0$
and $\mu_i \neq 0$.
Since Lemma~\ref{le:lambdapos} guarantees $\lambda_i > 0$ for all $i=1,2,\dots$,
the only condition is $\mu_i \neq 0$.

2.\ Let $\sigma(S)$ stand for the spectrum of $S$.
A well known result about compact operators, see, e.g., \cite[Theorem~4.24 (b)]{Rudin_FA_91},
implies that the set
$[\varepsilon,\infty) \cap \sigma(S)$ is finite
for any $\varepsilon > 0$.
The claimed statement now immediately follows, because
$\lambda_i = 1/\mu_i$ for all $\mu_i \neq 0$.

3.\ Since $S$ is selfadjoint in $V$,
the Courant--Fischer--Weyl min-max principle, see, e.g.,~\cite{Strang_Fix_analysis_FEM_08},
implies that
$$
  \mu_1 = \sup \{ a(S v,v) : \norm{v}_a = 1 \}
  = \sup\limits_{v\in V, v\neq 0} \frac{a(S v, v)}{\norm{v}_a^2}
  = \sup\limits_{v\in V, |v|_b\neq 0} \frac{|v|_b^2}{\norm{v}_a^2}
$$
is finite and it is the largest eigenvalue of the operator $S$.
Consequently,
\begin{equation}
\label{eq:lambda1inf}
  \lambda_1 = \mu_1^{-1}
= \inf\limits_{v\in V, |v|_b\neq 0} \frac{\norm{v}_a^2}{|v|_b^2}
\end{equation}
is the smallest eigenvalue of problem \eqref{eq:EP1}.
\end{proof}

Equality \eqref{eq:lambda1inf} immediately implies an abstract inequality of Friedrichs--Poincar\'e type, namely
\begin{equation}
\label{eq:absineq}
|v|_b \leq C_{ab} \norm{v}_a \quad \forall v \in V,
\end{equation}
where $C_{ab} = \lambda_1^{-1/2}$. This value of $C_{ab}$ is optimal, because $v=u_1$ yields equality in \eqref{eq:absineq}.

Further, eigenfunctions $u_i$, $i=1,2,\dots$, normalized as in \eqref{eq:normalization}
satisfy the Parseval's identity
$$
  |u_*|_b^2 = \sum\limits_{i=1}^\infty |b(u_*,u_i)|^2 \quad \forall u_* \in V.
$$
The proof is based on the splitting \eqref{eq:Vsplit} and follows the same steps as the proof
of Lemma~3.2 in \cite{Seb_Vejch_2sidedb_eigen_Fr_Poin_trace_14}.

\begin{theorem}
\label{th:estwb}
Let the solution operator $S$ defined in \eqref{eq:solop} be compact
and let $\lambda_i$, $i=1,2,\dots$, be eigenvalues of \eqref{eq:EP1}.
Let $u_* \in V$ and $\lambda_*\in \R$ be arbitrary.
Let $w\in V$ be such that
\begin{equation}
\label{eq:defw}
a(w,v) = a(u_*,v) - \lambda_* b(u_*,v) \quad\forall v \in V.
\end{equation}
If $|u_*|_b \neq 0$ then
$$
  \min\limits_i \left| \frac{\lambda_i-\lambda_*}{\lambda_i} \right|
   \leq \frac{|w|_b}{|u_*|_b}.
$$
\end{theorem}
\begin{proof}
It follows the same steps as the proof of Theorem~3.3 in \cite{Seb_Vejch_2sidedb_eigen_Fr_Poin_trace_14}.
\end{proof}

The following theorem is a consequence of Theorem~\ref{th:estwb} and it has been proved in~\cite[Theorem~3.4]{Seb_Vejch_2sidedb_eigen_Fr_Poin_trace_14} in a slightly different context. We repeat it here, because it provides an abstract enclosure on the principal eigenvalue of~\eqref{eq:EP1} and we use it below in Theorem~\ref{th:lowerbound} to obtain the lower bound for the general symmetric elliptic eigenvalue problem.
\begin{theorem}[Abstract complementarity estimate]
\label{le:abscompl}
Let $u_*\in V$, $|u_*|_b = 1$,
$\lambda_* \in \R$ be arbitrary and let $w\in V$
satisfy \eqref{eq:defw}.
Let the solution operator $S$ given by \eqref{eq:solop} be compact.
Let $\lambda_1$ be the smallest eigenvalue of \eqref{eq:EP1}
and
let the relatively closest eigenvalue to $\lambda_*$ be $\lambda_1$,
i.e., let
\begin{equation}
\label{eq:closest}
 \left|\frac{\lambda_1 - \lambda_*}{\lambda_1}\right|
\leq
  \left|\frac{\lambda_i - \lambda_*}{\lambda_i}\right|
\quad \forall i=1,2,\dots.
\end{equation}
Further, let $A \geq 0$ and $B \geq 0$ be such that
%
%
\begin{equation}
\label{eq:abscompl}
  \norm{w}_a \leq A + C_{ab} B  
  \quad\text{and}\quad
  B < \lambda_*
\end{equation}
where $C_{ab} = \lambda_1^{-1/2}$ is the optimal constant from \eqref{eq:absineq}.
Then
\begin{equation}
\label{eq:enclo_eig}
  \frac14 \left( -A + \sqrt{A^2 + 4(\lambda_* - B)} \right)^2 \leq \lambda_1. \\
\end{equation}
\end{theorem}

The crucial assumption of the above results is the compactness of the solution operator $S$. Therefore, we provide a useful tool that enables to utilize standard compactness results such as the Rellich and trace theorems.
\begin{lemma}
\label{le:compS}
Let $V$ be a Hilbert space with scalar product $a(\cdot,\cdot)$.
Let $H_1$ and $H_2$ be Hilbert spaces and let $\gamma_1 : V \rightarrow H_1$ and $\gamma_2: V \rightarrow H_2$ be two compact operators.
If the bilinear form $b$ is defined as
\begin{equation}
\label{eq:bH1H2}
  b(u,v) = b_1 (\gamma_1 u, \gamma_1 v) + b_2 (\gamma_2 u, \gamma_2 v)
\end{equation}
for some continuous bilinear forms $b_1$ and $b_2$ on $H_1$ and $H_2$, respectively,
then
the solution operator $S: V \rightarrow V$ defined in \eqref{eq:solop} is compact.
\end{lemma}
\begin{proof}
First, we define the solution operator $S_1 : H_1 \rightarrow V$ as follows.
Given $\varphi \in H_1$, we use the Riesz representation theorem to define a unique $S_1 \varphi \in V$ such that
$$
  a(S_1 \varphi, v) = b_1 (\varphi, \gamma_1 v) \quad\forall v \in V.
$$
Clearly, $S_1$ is linear and continuous. Since $\gamma_1$ is compact, the composition $S_1\gamma_1 : V \rightarrow V$ is compact as well.

In the same way, we can define the solution operator $S_2 : H_2 \rightarrow V$ and show the compactness of $S_2\gamma_2: V \rightarrow V$.
Now, using the definition \eqref{eq:solop}, we obtain
$$
  a(Su,v) = b(u,v) = b_1 (\gamma_1 u, \gamma_1 v) + b_2 (\gamma_2 u, \gamma_2 v) = a(S_1 \gamma_1 u, v) + a(S_2 \gamma_2 u, v)
$$
for all $u,v \in V$. Hence, $S = S_1\gamma_1 + S_2 \gamma_2$ and it is a compact operator, because it is a sum of two compact operators.
\end{proof}

%

\section{Symmetric elliptic eigenvalue problem}
\label{se:problem}

In what follows we consider elliptic eigenvalue problems formulated in a domain $\Omega$ with mixed boundary conditions of Dirichlet, Neumann, and/or Robin type. We seek the eigenvalue $\lambda_i \in \R$ and the corresponding eigenfunction $u_i \neq 0$ satisfying
\bse \label{eq:eigp_strong}
\begin{align}
  -\ddiv( \cA \nabla u_i ) + c u_i &= \lambda_i \beta_1 u_i  &\quad &\text{in }\Omega, \\
  (\cA \nabla u_i) \cdot \tn + \alpha u_i &= \lambda_i \beta_2 u_i &\quad &\text{on }\GammaN, \\
  u_i &= 0 &\quad &\text{on }\GammaD,
\end{align}
\ese
where $\GammaD$ and $\GammaN$ are portions of the boundary $\partial\Omega$ and $\tn$ stands for the unit outward facing normal to $\partial\Omega$.
Note that specific choices of parameters $\cA$, $c$, $\alpha$, $\beta_1$, $\beta_2$, and the sets $\GammaN$ and $\GammaD$ yield various well known types of eigenvalue problems, such as the Dirichlet or Neumann Laplace eigenvalue problem, or the Steklov eigenvalue problem.

Now, we show that the weak formulation of the eigenvalue problem \eqref{eq:eigp_strong} fits into the above abstract setting. We show compactness of the corresponding solution operator and, consequently, the existence of a countable sequence of eigenvalues $\lambda_i$ and the corresponding eigenfunctions $u_i$, $i=1,2,\dots$. We order them such that $\lambda_1 \leq \lambda_2 \leq \lambda_3 \leq \cdots$.

In order to introduce the weak formulation rigorously, we assume $\Omega$ to be a Lipschitz domain in $\R^2$.
The portions $\GammaD$ and $\GammaN$ are relatively open such that  $\partial\Omega = \oGammaD \cup \oGammaN$ and $\GammaD \cap \GammaN = \emptyset$.
Note that we admit cases with either $\GammaD$ or $\GammaN$ being empty.
We assume the diffusion coefficient to be a matrix function
$\cA \in [L^\infty(\Omega)]^{2\times 2}$, coefficients $c$ and $\beta_1$ to be in $L^\infty(\Omega)$, and coefficients $\alpha$ and $\beta_2$ in $L^\infty(\GammaN)$.
For technical reasons we also assume coefficients $\cA$, $c$, $\alpha$, $\beta_1$, and $\beta_2$ to be piecewise constant.
In order to guarantee the symmetry and ellipticity,
we assume coefficients $c$ and $\alpha$ to be nonnegative and the matrix $\cA$ to be symmetric and uniformly positive definite,
i.e. we assume existence of a constant $C>0$ such that
$$
  \b{\xi}^T \cA(x) \b{\xi} \geq C |\b{\xi}|^2 \quad\forall \b{\xi} \in \R^2
  \text{ and for almost all } x\in\Omega,
$$
where $|\cdot|$ stands for the Euclidean norm.

Let us note that the restriction to two spatial dimensions and the assumption of piecewise constant coefficients are due to technical reasons connected with the local flux reconstruction. These assumptions are needed for proofs of properties of the locally reconstructed flux and are not fundamental.

In what follows, we use the notation $(\cdot,\cdot)_Q$ for the $L^2(Q)$ scalar product
and $\norm{\cdot}_Q$ for the $L^2(Q)$-norm, where $Q$ typically stands for a subdomain of $\Omega$ or $\partial\Omega$. We also adopt the usual convention that if $Q = \Omega$ then we omit the subscript $\Omega$. Using this conventions, we define bilinear forms
\begin{align}
\label{eq:blf}
  a(u,v) &= (\cA \nabla u,\nabla v) + (c u, v) + (\alpha u, v)_\GammaN,
\\
\label{eq:blfb}
  b(u,v) &= (\beta_1 u, v) + (\beta_2 u, v)_\GammaN
\end{align}
and the usual space 
$$
  V = \{ v \in H^1(\Omega) : v = 0 \text{ on } \GammaD \}.
$$
We assume the form $a(u,v)$ to be a scalar product on $V$. This is the case if at least one of the following conditions is satisfied:
(a) $c > 0$ on a subset of $\Omega$ of positive measure,
(b) $\alpha > 0$ on a subset of $\GammaN$ of positive measure,
(c) measure of $\GammaD$ is positive.
In agreement with the notation introduced above, we denote by $\|{\cdot}\|_a$ and $|{\cdot}|_b$ the norm induced by $a({\cdot}, {\cdot})$ and the seminorm induced by $b({\cdot}, {\cdot})$, respectively.

In what follows, we consider the eigenvalue problem of finding $u_i \in V$, $u_i\neq 0$, and
$\lambda_i \in \R$ such that
\begin{equation}
\label{eq:eigp}
  a(u_i,v) = \lambda_i b(u_i,v) \quad \forall v \in V.
\end{equation}
The following theorem shows compactness of the corresponding solution operator.
Consequently, eigenproblem \eqref{eq:eigp} is well defined and posses all properties listed in Lemma~\ref{le:Sgamma} and Theorems~\ref{th:estwb} and \ref{le:abscompl}.
\begin{theorem}
Let bilinear forms $a(\cdot,\cdot)$ and $b(\cdot,\cdot)$ be defined by \eqref{eq:blf} and \eqref{eq:blfb} with the above listed requirements on the coefficients. Let $a(\cdot,\cdot)$ be a scalar product in $V$.
Then the solution operator $S$ defined by \eqref{eq:solop} is compact.
\end{theorem}
\begin{proof}
Notice that the form $b(u,v)$ defined in \eqref{eq:blfb} has the form \eqref{eq:bH1H2}.
Indeed, $H_1 = L^2(\Omega)$, $H_2 = L^2(\GammaN)$, bilinear forms $b_1$ and $b_2$ are just scalar products in $L^2(\Omega)$ and $L^2(\GammaN)$ with weighting functions $\beta_1$ and $\beta_2$, respectively. Operator $\gamma_1$ is the identity from the Sobolev space $V$ to $L^2(\Omega)$ and
it is compact due to Rellich theorem \cite[Theorem~6.3]{Adams_Sob_spaces_03}.
Operator $\gamma_2$ is the trace operator from $V$ to $L^2(\Gamma_N)$
and its compactness is proved in \cite[Theorem 6.10.5]{Kuf_John_Fucik_Function_spaces}; see also \cite{Biegert:2009}.
Thus, we can apply Lemma~\ref{le:compS} and conclude that the solution operator corresponding to problem
\eqref{eq:eigp} is compact.
\end{proof}

We discretize the eigenvalue problem \eqref{eq:eigp} by the standard conforming finite element method.
To avoid technicalities with curved elements, we assume the domain $\Omega$ to be polygonal
and consider a conforming (face-to-face) triangular mesh $\cT$. Formally, $\cT$ is a set of closed triangles. For technical reasons we assume that the mesh $\cT$ is a member of a shape regular family of triangulations $\cF$. Namely, we assume existence of a constant $C_{\mathrm{s}} > 0$ such that
\begin{equation} \label{shapereg}
\frac{h_{K}}{\rho_{K}} \leq C_{\mathrm{s}}
\quad \forall K \in
    \cT \text{ and } \forall \cT \in \cF,
\end{equation}
where $h_K = \operatorname{diam}(K)$ is the diameter of the triangle $K$ and $\rho_{K}$ denotes the diameter of the largest circle inscribed into $K$.

We note that the shape regularity~\eqref{shapereg} implies local quasi-uniformity. This can be easily shown by using \cite[Theorem 1]{Leng_Tang_Some_ineq_inra_simpl_faces_96}. Consequently, there exists a constant $C_{H} > 0$ such that for all meshes $\cT \in \cF$ and all elements $K \in \cT$ we have
\begin{equation} \label{locquasiuni}
  h_K \leq C_{H} h_{K'} \quad \forall K' \in \cT(\omega_K),
\end{equation}
where $\omK$ stands for the patch of elements sharing at least one vertex with $K$ and by $\cT(\omK)$ the set of elements in this patch. More precisely,
\begin{align}
 \label{eq:omK}
 \omK &= \interior\bigcup \{ K' \in \cT : K' \cap K \neq \emptyset \},
 \\
 \label{eq:TomK}
 \cT(\omK) &= \{ K' \in \cT : K' \subset \overline{\omega}_K \},
\end{align}
where $\interior$ denotes the interior of a domain.
Shape regularity also implies that the numbers of elements in patches $\cT(\omK)$ are uniformly bounded throughout the whole family $\cF$ as well as the numbers of patches an element is contained in.

%

Using the mesh $\cT$, we define the finite element space consisting of globally continuous and piecewise polynomial functions of degree at most $p$. We denote by $P_p(K)$ the space of polynomials of degree at most $p$ on the triangle $K \in \cT$ and set
\begin{equation}
\label{eq:VT}
  \VT = \{ \vT \in V : \vT|_K \in P_p(K),\ \forall K \in \cT\}.
\end{equation}
The discrete counterpart to the eigenvalue problem~\eqref{eq:eigp} reads:
%
Find $\lamT_i \in\R$ and $\uT_i \in \VT$, $\uT_i \neq 0$
such that
\begin{equation}
\label{eq:discreigp}
  a(\uT_i, \vT) = \lamT_i b(\uT_i, \vT ) \quad \forall \vT \in \VT.
\end{equation}

\section{Classical residual error estimator}
\label{se:classicalest}
In this section, we review the classical residual error estimator and show its local efficiency with respect to $\|w\|_a$, see \eqref{eq:defw}.
We consider an approximate eigenpair $\lamTf\in\R$ and $\uTf\in\VT$ for some fixed $i\geq 1$.
We denote by
\begin{equation}
\label{eq:R}
 R = - \ddiv (\cA\nabla\uTf) + c\uTf - \lamTf \beta_1 \uTf
\end{equation}
the classical residual of the approximation $\uTf$.
In order to define the jump residual, we introduce certain notation. We denote by $\cET$, $\cETI$, $\cETN$, and $\cETD$ the sets of edges in $\cT$, interior edges, edges on the Neumann part of the boundary, and edges on the Dirichlet part of the boundary, respectively.
For an edge $\Gamma \in \cET$, we consider an arbitrary but fixed unit normal vector $\tn_\Gamma$ and assume that for the boundary edges it coincides with the unit outward normal $\tn$.
Each interior edge $\Gamma \in \cETI$ lies between two elements and we denote by $K^-$ the one which the normal $\tn_\Gamma$ aims to. The other element is denoted by $K^+$ and we clearly have $\tn_{K^+} = \tn_\Gamma$ and $\tn_{K^-} = -\tn_\Gamma$. Thus, the jump of a function $\varphi$ over the edge $\Gamma \in \cETI$ is defined as the function $[\varphi]_\Gamma = \varphi|_{K^+} - \varphi|_{K^-}$ defined on $\Gamma$
and consequently $(\cA \nabla \uTf)|_{K^+} \cdot \tn_{K^+} + (\cA \nabla \uTf)|_{K^-} \cdot \tn_{K^-} = [\cA \nabla \uTf]_\Gamma \cdot \tn_\Gamma$.
Using this notation, we set
\begin{equation}
\label{eq:J}
   J|_\Gamma = \left\{ \begin{array}{ll}
    [\cA \nabla \uTf]_\Gamma \cdot \tn_\Gamma
       & \text{for } \Gamma \in \cETI,
    \\
    (\cA \nabla \uTf)|_\Gamma \cdot \tn - \lamTf \beta_2 \uTf|_\Gamma + \alpha \uTf|_\Gamma & \text{for } \Gamma \in \cETN,
    \\
    0 & \text{for } \Gamma \in \cETD.
    \end{array} \right.
\end{equation}

Quantities $R$ and $J$ define classical residual indicators $\eta_{R,K}$ of the error on elements $K\in\cT$ and the corresponding global error estimator $\eta_R$ as
\begin{equation}
\label{eq:etaRK}
  \eta_{R,K}^2 = h_K^2 \| R \|_K^2 + h_K \| J \|_{\partial K}^2 \quad \forall K\in\cT
  \quad\text{and}\quad
  \eta_R^2 = \sum_{K\in\cT} \eta_{R,K}^2.
\end{equation}

Notice that the equation \eqref{eq:defw} with bilinear forms defined in \eqref{eq:blf} and \eqref{eq:blfb}
and with $u_* = \uTf$ and $\lambda_* = \lamTf$ can be expressed as
\begin{equation}
  \label{eq:wRJ}
  a(w,v) = a(\uTf,v) - \lamTf b(\uTf,v) = \sum_{K\in\cT} (R,v)_K + \sum_{\Gamma\in\cET} (J,v)_\Gamma
  \quad\forall v \in V.
\end{equation}
This identity, together with the standard technique of bubble functions \cite{Verfurth:1994}, see also \cite{AinOde:2000},
can be used to prove the efficiency of the classical residual estimator with respect to the energy norm of the residual representative $w$. Since we assume piecewise constant data, there are no oscillation terms.
\begin{lemma}[Efficiency of the classical residual indicators]
\label{le:effclasres}
Let $\cF$ be a shape regular family of triangulations and let $\cT \in \cF$.
Let $\lamTf\in\R$, $\uTf\in\VT$ be an arbitrary approximation of the eigenpair $\lambda_i \in\R$, $u_i\in V$ of \eqref{eq:eigp}.
Let $w \in V$ be given by \eqref{eq:defw} with bilinear forms \eqref{eq:blf} and \eqref{eq:blfb} and with $\lambda_* = \lamTf$ and $u_* = \uTf$. Let $\eta_{R,K}$ stand for the classical residual error indicators \eqref{eq:etaRK}. Then there exists a constant $C>0$ uniform over the family $\cF$ such that
\begin{equation}
  \label{eq:effetaRK}
  \eta_{R,K} \leq C \norm{w}_{a,\omK}.
\end{equation}
\end{lemma}
\begin{proof}
Let $\psi_K$ be an interior bubble function on the element $K$, i.e. $\psi_K > 0$ in the interior of $K$ and vanishes on its boundary $\partial K$. For example, it can be a cubic polynomial. We define a weighted norm $\norm{\varphi}_{\psi_K}^2 = (\psi_K \varphi, \varphi)_K$. Based on the equivalence of norms on finite dimensional spaces and the inverse inequality, there exist constants $C_1>0$ and $C_2 >0$ such that
\begin{equation}
  \label{eq:normequiv}
  \norm{\varphi}_K \leq C_1 \norm{\varphi}_{\psi_K} 
  \quad\text{and}\quad
  \norm{\psi_K \varphi}_{H^1(K)} \leq C_2 h_K^{-1} \norm{\varphi}_K \quad \forall \varphi \in P_p(K),
\end{equation}
see e.g. \cite[Theorem~2.2]{AinOde:2000} for details.
Since $\psi_K R$ is supported in $K$ only and since $\psi_K R \in V$, we can use it as a test function in \eqref{eq:wRJ} and obtain
$ a_K(w,\psi_K R) = a(w,\psi_K R) = \norm{ R }^2_{\psi_K}$,
where $a_K$ is the restriction of $a$ on $K$, i.e. $a_K$ is defined as in \eqref{eq:blf}, but all integrals over $\Omega$ are replaced by integrals over $K$.
Consequently, since $R|_K \in P_p(K)$, we have
$$
  \norm{R}_K^2 \leq C_1^2 \norm{R}_{\psi_K}^2 = C_1^2 a_K(w,\psi_K R)
  \leq C_1^2 \norm{ w }_{a,K} \norm{\psi_K R}_{a,K}
  \leq C h_K^{-1} \norm{ w }_{a,K} \norm{R}_K,
$$
where the last inequality follows from the equivalence of the energy and $H^1$ norms and \eqref{eq:normequiv}.
Thus,
\begin{equation}
  \label{eq:normRest}
  \norm{R}_K \leq C h_K^{-1} \norm{ w }_{a,K}.
\end{equation}

Similarly, we can bound the jumps $J$. Let $\Gamma$ be an edge, $\omega_\Gamma = \interior\bigcup \{ K \in \cT : \Gamma \subset \partial K \}$ and $\cT(\omega_\Gamma) = \{ K \in \cT : K \subset {\overline{\omega}}_\Gamma \}$. Let $\psi_\Gamma$ be an edge bubble function, i.e. $\psi_\Gamma > 0$ in $\omega_\Gamma$ and vanishes on $\partial \omega_\Gamma$. For example, it can be a quadratic polynomial on each element $K \in \cT(\omega_\Gamma)$.
As above, we introduce a weighted norm $\norm{\varphi}_{\psi_\Gamma}^2 = (\psi_\Gamma \varphi,\varphi)_\Gamma$ and consider estimates
\begin{align}
  \label{eq:psigammanorm}
  \norm{\varphi }_\Gamma &\leq C_3 \norm{ \varphi }_{\psi_\Gamma}
  \quad \forall \varphi\in P_p(\Gamma),
  \\
  \label{eq:psigammanorm2}
  \norm{\psi_\Gamma \varphi }_K &\leq C_4 h_K^{1/2}\norm{ \varphi }_\Gamma
  \quad \forall \varphi\in P_p(K),
  \\
  \label{eq:psigammanorm3}
  \norm{\psi_\Gamma \varphi }_{H^1(K)} &\leq C_5 h_K^{-1/2}\norm{ \varphi }_\Gamma
  \quad \forall \varphi\in P_p(K),
\end{align}
where $K$ is any element from $\cT(\omega_\Gamma)$, see e.g. \cite[Theorem~2.4]{AinOde:2000} for details.

The jumps $J$ are defined on edges $\Gamma \in \cET$ and $J|_\Gamma \in P_p(\Gamma)$. For the purpose of this prove we extend $J$ polynomially into the interior of all elements. This extension can be arbitrary but fixed and satisfying $J|_K \in P_p(K)$.
If $\Gamma \in \cETI \cup \cETN$ then $\psi_\Gamma J \in V$ and we can use it as a test function in \eqref{eq:wRJ} and obtain $a_{\omega_\Gamma}(w,\psi_\Gamma J) = a(w,\psi_\Gamma J) = (\psi_\Gamma R, J)_{\omega_\Gamma} + (\psi_\Gamma J, J)_\Gamma$, because $\psi_\Gamma J$ is supported in $\omega_\Gamma$.
Using \eqref{eq:psigammanorm}, we have
$$
  \norm{J}_\Gamma^2 \leq C_3^2 \norm{ J }_{\psi_\Gamma}^2
  = C_3^2 \left[ a_{\omega_\Gamma}(w,\psi_\Gamma J) - (\psi_\Gamma R, J)_{\omega_\Gamma} \right]
  \leq C_3^2 \left[ \norm{ w }_{a,\omega_\Gamma} \norm{\psi_\Gamma J}_{a,\omega_\Gamma}
    + \norm{R}_{\omega_\Gamma} \norm{ \psi_\Gamma J }_{\omega_\Gamma} \right].
$$
Now we use estimates $\norm{ \psi_\Gamma J }_{a,\omega_\Gamma} \leq C h_\Gamma^{-1/2} \norm{ J }_\Gamma$
and $\norm{ \psi_\Gamma J }_{\omega_\Gamma} \leq C h_\Gamma^{1/2} \norm { J }_\Gamma$,
which follow from the equivalence of the energy norm with the $H^1$ norm and from \eqref{eq:psigammanorm2}--\eqref{eq:psigammanorm3},
and we obtain
$$
  \norm{J}_\Gamma \leq C \left[ h_\Gamma^{-1/2} \norm{w}_{a,\omega_\Gamma}
      + h_\Gamma^{1/2} \norm{ R }_{\omega_\Gamma} \right].
$$
Estimate \eqref{eq:normRest} finally yields
\begin{equation}
  \label{eq:normJest}
  \norm{J}_\Gamma \leq C h_\Gamma^{-1/2} \norm{ w }_{a,\omega_\Gamma}.
\end{equation}
Combination of \eqref{eq:normRest} and \eqref{eq:normJest} finishes the proof.
\end{proof}

Note that the approximate eigenpair $\lamTf$, $\uTf$ in Lemma~\ref{le:effclasres} need not be given by \eqref{eq:discreigp}.

\section{Local flux reconstruction and the error estimator}
\label{se:errest}

In this section we define the error indicators and the corresponding error estimator based on an $\Hdiv$ reconstruction of the flux $\cA \nabla \uTf$.
In contrast to the classical residual estimator, the estimator based on this flux reconstruction provides a fully computable upper bound on $\| w \|_a$. This enables to compute lower bounds on the exact eigenvalues $\lambda_i$, see Theorem~\ref{th:lowerbound} below.
Technically, we use the flux reconstruction proposed in \cite{BraSch:2008} for source problems and use it for eigenvalue problems of type \eqref{eq:eigp}. This approach is local and efficient, because it is based on solving small problems on patches of elements.

We continue to consider $\lamTf\in\R$ and $\uTf\in\VT$, $i\geq 1$, to be a fixed approximate eigenpair.
First, we introduce the error indicators $\eta_K$ and the resulting error estimator $\eta$ as
\begin{equation}
\label{eq:etaK}
  \eta_K = \| \nabla \uTf - \cA^{-1}\tqT \|_{\cA,K} \quad \forall K\in\cT
  \quad\text{and}\quad
  \eta^2 = \sum_{K\in\cT} \eta_K^2,
\end{equation}
where the norm is defined by $\| \tq \|_{\cA,K}^2 = (\cA \tq, \tq)_K$ and $\tqT \in \Hdiv$ is the local flux reconstruction.

The description of the local flux reconstruction $\tqT$ and proofs of its properties are technical and they were inspired mainly by works \cite{Ern_Voh_adpt_IN_13} and \cite{Dol_Ern_Vohr_hp_refinement_strategies_polyn_rob_AEE}.
The flux reconstruction $\tqT$ is naturally defined in the Raviart--Thomas finite element spaces. Therefore, we first review their properties, see e.g.~\cite{BreFor:1991} and \cite{Quar_Val_Num_appr_PDE_94},
and introduce the notation. Then we define the flux reconstruction $\tqT$ and prove several lemmas and a theorem.

The Raviart--Thomas space of order $p$ is defined on the mesh $\cT$ as
\begin{equation}\label{RTN}
\bWT = \left\{  \bwT \in \Hdiv:
    \bwT|_K \in \RT(K) \quad \forall K \in \cT
  \right \},
\end{equation}
where $\RT(K) = [P_p(K)]^2 \oplus \boldsymbol{x}P_p(K)$ and $\boldsymbol{x} = (x_1,x_2)$ is the vector of coordinates.
Functions $\bwT \in \bWT$ have continuous normal components across the internal edges
and $\ddiv \bwT|_K \in P_p(K)$ for all $K \in \cT$.
In addition, the normal components of $\bwT$ on edges $\Gamma$ span the whole space $P_p(\Gamma)$ of polynomials of degree at most $p$ on $\Gamma$ and we have
\begin{equation}
\label{eq:normcomp=Pp}
\left\{ \bwT|_\Gamma {\cdot} \tn_\Gamma : \bwT \in \bWT \right\} = P_{p}(\Gamma)
\quad \forall \Gamma \in \cET.
\end{equation}

We introduce the notation for vertices (nodes) of the mesh $\cT$.
Let $\NT$ denote the set of all vertices in $\cT$. The subsets of those lying on $\oGammaD$, on $\GammaN$, and in the interior of $\Omega$ are denoted by $\NTD$, $\NTN$, and $\NTI$, respectively. Notice that if a vertex is located at the interface between the Dirichlet and Neumann boundary, it is not in $\NTN$, but only in $\NTD$.
We also denote by $\NTK$ and $\cETK$ the sets of three vertices and three edges of the element $K$, respectively.

We construct the flux reconstruction $\tqT \in \bWT$ by solving local Neumann and Neumann/Dirichlet mixed finite element problems defined on patches of elements sharing a given vertex.
Let $\ta \in \NT$ be an arbitrary vertex, we denote by $\HF$ the standard piecewise linear and continuous hat function associated with $\ta$. This function vanishes at all vertices of $\cT$ except of $\ta$, where it has value 1. Note that $\HF \in \VT$ for vertices $\ta \in \NTI \cup \NTN$, but $\HF \not\in \VT$ for $\ta \in \NTD$.
Further, let $\Ta = \{ K \in \cT : \ta \in K \}$ be the set of elements sharing the vertex~$\ta$ and
$\Oma = \interior\bigcup \{ K: K \in \Ta \}$ the patch of elements sharing the vertex $\ta$.
We denote by $\cEaI$ the set of interior edges in the patch $\Oma$,
by $\cEaBE$ the set of those edges on the boundary $\poma$ that do not contain $\ta$,
and by $\cEaBD$ and $\cEaBN$ the sets of edges on the boundary $\poma$ with an end point at $\ta$ lying either on $\GammaD$ or on $\GammaN$, respectively. Note that sets $\cEaBD$ and $\cEaBN$ can be nonempty only if $\ta \in \NTD \cup \NTN$, i.e. for boundary patches.

We also introduce auxiliary quantities
\begin{align}
\label{eq:ra}
\rTa &= \lamTf \beta_1 \HF  \uTf - c \HF \uTf - (\cA \nabla \HF) \cdot \nabla \uTf,
\\
\label{eq:ga}
\gTa &= \lamTf \beta_2 \HF \uTf - \alpha \HF \uTf.
\end{align}
Note that these quantities are defined in such a way that
\begin{equation}
  \label{eq:ragN}
  a(\uTf, \HF) - \lamTf b(\uTf,\HF) = - \int_{\Oma} \rTa \dx - \int_{\cEaBN} \gTa \dx[s]
  \quad\forall \ta\in\NT.
\end{equation}

The local flux reconstruction is defined in the Raviart--Thomas spaces on patches $\Oma$ with suitable
boundary conditions. We introduce the space
\begin{multline}
\label{eq:Wa0}
\bWa^0 = \left\{
    \bwT \in \Hdiv[\Oma] : \bwT|_K \in \RT(K) \quad\forall K \in \Ta
    \right. \\ \left.
    \quad\text{and}\quad
    \bwT \cdot \tn_{\Gamma} =
      0 \text{ on edges } \Gamma \in \cEaBE \cup \cEaBN
  \right\},
\end{multline}
and the affine set
\begin{multline}
  \label{eq:Wa}
  \bWa = \left\{
    \bwT \in \Hdiv[\Oma] : \bwT|_K \in \RT(K) \quad\forall K \in \Ta, \quad
    \bwT \cdot \tn_{\Gamma} = 0 
    \right. \\ \left. 
    \text{ on edges } \Gamma \in \cEaBE 
    \text{ and }
    \bwT \cdot \tn_{\Gamma} = \Pi_{\Gamma}(\gTa) \text{ on edges } \Gamma \in \cEaBN
  \right\}.
\end{multline}
The symbol $\Pi_{\Gamma}$ stands for the $L^2(\Gamma)$-orthogonal projection onto the space $P_p(\Gamma)$ of polynomials of degree at most $p$ on the edge $\Gamma$.
We also define the space $P_p(\Ta) = \{ \vT \in L^2(\Oma) : \vT|_K \in P_p(K) \ \forall K\in\Ta \}$ of piecewise polynomial and in general discontinuous functions.
Further, we introduce the space
\begin{equation} \label{eq:Ppast}
  \PpTast = \left\{ \begin{array}{ll}
    \{ \vT \in P_p(\Ta): \int_{\Oma} \vT \dx = 0 \}, \quad \text{for} \ \ta \in \NTI \cup \NTN, \\
    P_p(\Ta) , \quad \text{for} \ \ta \in \NTD.
    \end{array} \right.
\end{equation}

Using these spaces, we define the flux reconstruction $\tqT \in \bWT$ as the sum
\begin{equation} \label{eq:defq}
  \tqT = \sum_{\ta \in \NT} \tqa,
\end{equation}
where $\tqa \in \bWa$ together with $\da \in \PpTast$ solves the mixed finite element problem
\bse \label{min_prob_qa}
\begin{alignat}{2}
       (\cA^{-1} \tqa, \bwT)_{\Oma} - (\da, \ddiv \bwT)_{\Oma} &= (\HF \nabla \uTf, \bwT)_{\Oma}
         &\quad& \forall \bwT \in \bWa^0, \label{min_prob_qa_1}\\
       - (\ddiv \tqa, \vT)_{\Oma} &= (\rTa, \vT)_{\Oma} &\quad &\forall \vT \in \PpTast. \label{min_prob_qa_2}
\end{alignat}
\ese
Let us note that this mixed finite element problem is equivalent to the minimization of
$\left\|\HF \cA^{\frac{1}{2}} \nabla \uTf - \cA^{-\frac{1}{2}} \tsa \right\|_{\Oma}$
over all $\tsa \in \bWa$ satisfying the constraint $-\ddiv \tsa = \Pi_{p}(\rTa)$ in $\Oma$,
where $\Pi_{p}$ denotes the $L^2(\Oma)$-orthogonal projection onto $P_p(\Ta)$.



In order to show that the fluxes $\tqa$ are well defined and that problem \eqref{min_prob_qa} is uniquely solvable, we recall a regularity result for the solution of the Poisson equation in polygonal domains, see e.g.~\cite{Grisvard_prob_lim_polyg_86}. More precisely, we show that the gradient of the solution of the Poisson equation with Dirichlet,  Neumann, or mixed boundary conditions lies in $L^p(\Omega)$ for some $p > 2$. 
This result together with the continuous inf-sup condition enables us to present the validity of the associated discrete inf-sup condition.
\begin{lemma}
\label{le:regul}
Let $\Omega$ be a polygon. Let $\psi$ be a solution of the Poisson equation with a right-hand side in $L^2(\Omega)$ and with Dirichlet, Neumann, or mixed boundary conditions. Then $\nabla \psi \in [L^p(\Omega)]^2$ for certain $p > 2$.
\end{lemma}
\begin{proof}
The solution $\psi$ is well known to be smooth with the exception of 
neighbourhoods of vertices of the polygon $\Omega$ and points where the type of boundary conditions changes.
Let us consider such a point and the angle $0 < \theta \leq 2\pi$ between the two adjacent sides. 
In the vicinity of this point, the solution $\psi$ lies in $H^s(\Omega)$ for 
all $s < 1+\pi/\theta$ provided there is the same type of boundary conditions prescribed on both adjacent sides \cite{Grisvard_prob_lim_polyg_86}.
If the Dirichlet boundary condition is prescribed on one of these sides and the Neumann boundary condition on the other side then $\psi$ lies in $H^s(\Omega)$ for all $s < 1+\pi/(2\theta)$, see \cite{Grisvard_prob_lim_polyg_86}.

In particular, we conclude that $\psi \in H^s(\Omega)$ for $s \in (1,5/4)$ in all cases.
Thus, $\nabla \psi \in [H^t(\Omega)]^2$ for $t \in (0,1/4)$.
To finish the proof, we notice that 
the Sobolev embedding theorem for fractional orders in two-dimensions yields 
$H^t(\Omega) \subseteq L^{t^*}(\Omega)$
for $t^* = 2/(1-t)$. 
Since $t \in (0,1/4)$, we have $t^* > 2$ and the proof is complete.
\end{proof}

The following lemma introduces the inf-sup condition, which is formulated in terms of spaces
%
\begin{align*}
\HdivO[\Oma] &= \left\{ \bw \in \Hdiv[\Oma] : \bw \cdot \tn = 0 \text{ on edge } 
  \Gamma \in \cEaBE \cup \cEaBN
 \right\},
\\
Z^*(\Oma) &= \left\{ 
  \begin{array}{ll}
    \{ v \in L^2(\Oma): \int_{\Oma} v \dx = 0 \}, \quad \text{for} \ \ta \in \NTI \cup \NTN, \\
    L^2(\Oma), \quad \text{for} \ \ta \in \NTD.
  \end{array}
\right.
\end{align*}
\begin{lemma} \label{lem_inf-sup_continuous}
   Let $\bW = \HdivO[\Oma] \cap [L^s(\Oma)]^2$, $s > 2$. 
   Then there exists a constant $\beta > 0$ such that
   \begin{align}\label{inf-sup_continuous}
      \sup_{\ttau\in \bW}
    \frac{(\varphi, \ddiv \ttau)_{\Oma}}{\|\ttau\|_{\bW}}
    \ge \beta \|\varphi\|_{\Oma} \quad
    \forall \varphi \in Z^*(\Oma).
   \end{align}
\end{lemma}

\begin{proof}
  Let $\ta \in \NT$ be a vertex and $\varphi\in Z^*(\Oma)$ arbitrary, but fixed. 
  Consider Poisson equation $-\Delta \psi = \varphi$ in $\Oma$ with
  homogeneous Neumann boundary conditions on edges $\Gamma \in \cEaBE \cup \cEaBN$
  and with homogeneous Dirichlet boundary conditions on edges $\Gamma \in \cEaBD$.
  Notice that if $\ta \in \NTD$ then the boundary conditions are mixed Dirichlet--Neumann
  and the Poisson problem is well posed.
  If $\ta \in \NTI \cup \NTN$ then the boundary conditions are pure Neumann,
  but the problem is still solvable, because $\int_{\Oma} \varphi \dx = 0$.
  Among all solutions of this Neumann problem, we consider the one satisfying $\int_{\Oma} \psi \dx = 0$.

%
%
   Lemma~\ref{le:regul} implies $\nabla\psi \in \bW$ and we can take
   $\ttau = - \nabla \psi$ in the left-hand side of \eqref{inf-sup_continuous}. We obtain
   \begin{multline*}
   \sup_{\ttau\in \bW} \frac{(\varphi, \ddiv \ttau)_{\Oma}}{\|\ttau\|_{\bW}}
    \geq \frac{\|\varphi\|_{\Oma}^2}{\|\varphi\|_{\Oma} + \|\nabla \psi\|_{\Oma} + \|\nabla \psi\|_{L^s(\Oma)}} 
    \\
    \geq \frac{\|\varphi\|_{\Oma}^2}{\|\varphi\|_{\Oma} + 2\|\nabla \psi\|_{\Oma}} 
    \geq \frac{1}{1 + 2C} \|\varphi\|_{\Oma},
   \end{multline*}
   where we have used the stability of the Poisson problem $\|\nabla \psi\|_{\Oma} \leq C \|\varphi\|_{\Oma}$.
\end{proof}

\begin{lemma} \label{lem_inf-sup_discrete}
    There exists $\beta > 0$ such that the following discrete inf-sup condition holds:
    \begin{equation}\label{disc_inf-sup}
       \sup_{\ttauT\in \bWa^0}
       \frac{ (\varphiT, \ddiv \ttauT)_{\Oma}}{\|\ttauT\|_{\HdivO[\Oma]}}
       \ge \beta \|\varphiT\|_{L^2(\Oma)}
       \quad
       \forall \varphiT \in \PpTast.
    \end{equation}
\end{lemma}
\begin{proof}
   The statement follows from Lemma~\ref{lem_inf-sup_continuous}, \cite[Proposition 5.4.3]{Boffi_Brezzi_Fortin_MFEMs_applications_13} and properties of the interpolation operator shown in \cite[Proposition 2.5.2]{Boffi_Brezzi_Fortin_MFEMs_applications_13}.
\end{proof}

\begin{lemma}
    Problem~\eqref{min_prob_qa} has a unique solution.
\end{lemma}

\begin{proof}
    Let $\ttqa \in \bWa$ be arbitrary and let $\tqaO = \tqa - \ttqa$. Then the problem~\eqref{min_prob_qa} is equivalent to the problem of finding $\tqaO \in \bWa^0$ and $\daO \in \PpTast$ such that
    \bse \label{min_prob_qaO}
%
%
    \begin{align}
       (\cA^{-1} \tqaO, \bwT)_{\Oma} - (\daO, \ddiv \bwT)_{\Oma} &= (\HF \nabla \uTf, \bwT)_{\Oma} - (\cA^{-1} \ttqa, \bwT)_{\Oma}
       \label{min_prob_qaO_1}
       \\
       - (\ddiv \tqaO, \vT)_{\Oma} &= (\rTa, \vT)_{\Oma} + (\ddiv \ttqa, \vT)_{\Oma} 
       \label{min_prob_qaO_2}
    \end{align}
    \ese
    for all $\bwT \in \bWa^0$ and all $\vT \in \PpTast$.
    %


    Equations \eqref{min_prob_qaO} correspond to a linear algebraic system with a square matrix. Therefore, it is sufficient to show that $\HF \nabla \uTf = 0$, $\ttqa =0$ together with $\rTa = 0$ implies that the only possible solution of system \eqref{min_prob_qaO} is $\tqaO = \tO$ and $\daO=0$. Taking $\bwT = \tqaO$ in~\eqref{min_prob_qaO_1} and $\vT = \daO$ in~\eqref{min_prob_qaO_2} and subtracting the equation \eqref{min_prob_qaO_2} from \eqref{min_prob_qaO_1} yields $( \mathcal{A}^{-1} \tqaO, \tqaO )_{\Oma} = 0$ and thus the only possibility is $\tqaO = \tO$. Now, equality \eqref{min_prob_qaO_1} implies
    \begin{equation}
      \label{eq:dadivv}
      (\daO, \ddiv \bwT)_{\Oma} = 0 \quad \forall\bwT \in \bWa^0.
    \end{equation}
    By combining \eqref{eq:dadivv} and \eqref{disc_inf-sup} with $\varphiT = \daO$, we conclude that the only possibility is $\daO = 0$.
%
\end{proof}

Next lemma shows that we can actually test \eqref{min_prob_qa_2} by any polynomial.
\begin{lemma}
Let $\tqa \in \bWa$ and $\da \in \PpTast$ be a solution of problem \eqref{min_prob_qa}. Then
\begin{equation}
\label{eq:testPp}
- (\ddiv \tqa, \vT)_{\Oma} = (\rTa, \vT)_{\Oma} \quad \forall \vT \in \PpT.
\end{equation}
\end{lemma}
\begin{proof}
Notice that if $\ta \in \NTD$ then there is nothing to prove due to definition~\eqref{eq:Ppast}.
If $\ta \in \NTI \cup \NTN$ then we can use $\HF$ as a test function in \eqref{eq:discreigp}.
Consequently, identity \eqref{eq:ragN} and definition \eqref{eq:Wa} imply
\begin{equation}
\label{eq:Neumequilib}
  (\rTa,1)_{\Oma} = - (\gTa, 1)_{\GammaN\cap\partial\Oma}
  = - (\tqa \cdot \tn_{\partial\Oma}, 1)_{\partial\Oma}
  = -(\ddiv \tqa, 1)_{\Oma}.
\end{equation}
\end{proof}

Let us note that for $\ta \in \NTI \cup \NTN$, problem~\eqref{min_prob_qa}
corresponds to a pure Neumann problem for $\da$.
This problem is solvable, because the corresponding equilibrium condition is exactly \eqref{eq:Neumequilib}.
In addition, its solution is unique thanks to the fact that
the space $\PpTast$ does not contain constant functions. For $\ta \in \NTD$, problem~\eqref{min_prob_qa} corresponds to a well posed Dirichlet--Neumann problem
and the space $\PpTast$ contains constant functions.


Now, we present an important property of the introduced flux reconstruction.
\begin{lemma}
\label{le:etaR=0}
Let $\tqT \in \bWT$ be given by \eqref{eq:defq} and problems \eqref{min_prob_qa}.
Then
\begin{align}
 \label{eq:etaR=0}
 \lamTf \beta_1 \uTf - c \uTf + \ddiv \tqT &= 0
 \quad\text{a.e. in }\Omega,
 \\
 \label{eq:etaN=0}
 \alpha \uTf - \lamTf \beta_2 \uTf + \tqT \cdot \tn &= 0
 \quad\text{a.e. on }\GammaN.
\end{align}
\end{lemma}
\begin{proof}
   Let us set $R_\tq = - \ddiv \tqT + c \uTf - \lamTf \beta_1 \uTf$.
   Clearly, $R_\tq \in P_p(K)$ for all $K\in\cT$.
   Using the decomposition of unity $\sum_{\ta \in \NTK} \HF = 1$
   and \eqref{eq:testPp}, we obtain
   \begin{multline*} 
      \norm{R_\tq}^2_K
      = \sum_{\ta \in \NTK} \left( -\ddiv \tqa + c \uTf \HF - \lamTf \beta_1 \uTf \HF,
               R_\tq \right)_K
      \\
      = \sum_{\ta \in \NTK} \left( - \ddiv \tqa - \rTa - (\cA \nabla\HF)\cdot\nabla\uTf,
               R_\tq \right)_K
      = - \sum_{\ta \in \NTK} \left( (\cA \nabla\HF)\cdot \nabla\uTf, R_\tq \right)_K = 0.
   \end{multline*}
   Thus, $R_\tq$ vanishes almost everywhere in all elements $K\in\cT$ and, hence, in $\Omega$.

  To prove the second statement, we set $J_\tq =  \alpha \uTf - \lamTf \beta_2 \uTf + \tqT \cdot \tn$. Let $\Gamma \in \cETN$ be an arbitrary edge on the Neumann boundary. Clearly, $J_\tq \in P_p(\Gamma)$.
  The decomposition of unity $\sum_{\ta \in \NTGa} \HF = 1$ with $\NTGa$ being the set of the two end-points of the edge $\Gamma$ and definition~\eqref{eq:Wa} then give
  \begin{multline}\label{eval_trace_Neum_indic}
    \norm{ J_\tq }_\Gamma^2
    = \sum_{\ta \in \NTGa} (\alpha \HF\uTf - \lamTf \beta_2 \HF\uTf + \tqa {\cdot} \tn_\Gamma, J_\tq)_\Gamma
    \\
    = \sum_{\ta \in \NTGa} (\alpha \HF\uTf - \lamTf \beta_2 \HF\uTf + \Pi_\Gamma(\gTa), J_\tq)_\Gamma
    = 0.
  \end{multline}
  Thus, $J_\tq$ vanishes almost everywhere on $\Gamma$ and, hence, on $\GammaN$.
\end{proof}

The following lemma shows the distinctive feature of the error estimator $\eta$ given by \eqref{eq:etaK}.
It provides a guaranteed and fully computable upper bound on $\| w \|_a$, see the definition \eqref{eq:defw}.
In contrast, the classical residual estimator \eqref{eq:etaRK} provides an upper bound on $\| w \|_a$ up to an unknown multiplicative constant.


\begin{lemma} \label{le:normw}
Let the flux reconstruction $\tqT \in \bWT$ be given by \eqref{eq:defq} and problem \eqref{min_prob_qa}. 
Let $\lamTf \in \R$ and $\uTf \in \VT$ satisfy \eqref{eq:discreigp}.
Further, let $w \in V$ be given by \eqref{eq:defw} with bilinear forms defined in \eqref{eq:blf}--\eqref{eq:blfb}
and with $\lambda_* = \lamTf$ and $u_* = \uTf$.
Let $\eta = \norm{\nabla \uTf - \cA^{-1} \tqT}_{\cA}$ be given by \eqref{eq:etaK}.
Then
$$
  \norm{w}_a \leq \eta.
$$
\end{lemma}
\begin{proof}
Using test function $v = w$ in \eqref{eq:defw}, expanding the bilinear forms on the right-hand side according to \eqref{eq:blf} and \eqref{eq:blfb}, and applying the divergence theorem, we end up with the following expression
$$
  \norm{w}_a^2 = ( \cA \nabla \uTf - \tqT, \nabla w ) + (c \uTf - \lamTf \beta_1 \uTf - \ddiv\tqT, w)
    + ( \alpha \uTf - \lamTf \beta_2 \uTf + \tqT \cdot \tn, w)_\GammaN.
$$
The last two terms vanish due to \eqref{eq:etaR=0} and \eqref{eq:etaN=0} and hence we can estimate $\norm{w}_a^2$ by the Cauchy--Schwarz inequality as
$$
  \norm{w}_a^2 \leq \norm{\nabla \uTf - \cA^{-1} \tqT}_{\cA} \norm{\nabla w}_{\cA}
  \leq \norm{\nabla \uTf - \cA^{-1} \tqT}_{\cA} \norm{w}_{a}.
$$
\end{proof}

We conclude this section by a theorem which shows that the error estimator \eqref{eq:etaK} can be used to compute a lower bound on the principal eigenvalue $\lambda_1$. This is only possible thank to the guaranteed upper bound provided by Lemma~\ref{le:normw}.
\begin{theorem}[Lower bound on the principal eigenvalue]
\label{th:lowerbound}
Let $\lamT_1 \in \R$ and $\uT_1 \in \VT$, $|\uT_1|_b = 1$,
be the approximate principal eigenvalue and the corresponding eigenfunction given by \eqref{eq:discreigp}
and let the error estimator $\eta$ be given by \eqref{eq:etaK}.
Let $\lambda_1$ be the smallest eigenvalue of \eqref{eq:eigp}
and let it satisfy
\begin{equation}
  \label{eq:closestcond}
  \lamT_1 \leq 2 \left( \lambda_1^{-1} + \lambda_2^{-1} \right)^{-1}.
\end{equation}
Then
\begin{equation}
  \label{eq:underlinelambda}
  {\underline\lambda}_1^{\cT} \leq \lambda_1,
  \quad\text{where }
  {\underline\lambda}_1^{\cT} = \frac14 \left( -\eta + \sqrt{\eta^2 + 4\lambda_1^{\cT}} \right)^2.
\end{equation}
\end{theorem}
\begin{proof}
The proof follows straightforwardly from Theorem~\ref{le:abscompl} and Lemma~\ref{le:normw}.
Indeed, it is easy to verify that condition \eqref{eq:closestcond} is equivalent to the relative closeness assumption \eqref{eq:closest}.
Lemma~\ref{le:normw} implies that quantities $A$ and $B$ in \eqref{eq:abscompl} are $A=\eta$ and $B=0$ and estimate \eqref{eq:enclo_eig} then becomes \eqref{eq:underlinelambda}.
%
\end{proof}

The lower bound on the principal eigenvalue enables to derive a lower bound on any other eigenvalue as follow.

\begin{theorem}[Lower bounds on the other eigenvalues]
\label{th:lowerboundi}
Let $\lamTf\in\R$ and $\uTf\in\VT$ with $|\uTf|_b = 1$ and $i \geq 2$ be an approximate eigenpair satisfying \eqref{eq:discreigp}. Let the error estimator $\eta$ be given by \eqref{eq:etaK}
and let $\underline{\lambda}_1 \leq \lambda_1$ be a lower bound on the first eigenvalue.
If the approximation $\lamTf$ satisfies
\begin{equation}
  \label{eq:closestcondi}
  \lamTf \leq
  2 \left( \lambda_i^{-1} + \lambda_{i+1}^{-1} \right)^{-1}
\end{equation}
then
\begin{equation}
  \label{eq:underlinelambdai}
  {\underline\lambda}_i^{\cT} \leq \lambda_i,
  \quad\text{where }
  {\underline\lambda}_i^{\cT} =
     \lamTf \left( 1 + \underline{\lambda}_1^{-1/2} \eta \right)^{-1},
  \quad
  i \geq 2.
\end{equation}
\end{theorem}
\begin{proof}
Conditions \eqref{eq:closestcondi} and $\lambda_i \leq \lamTf$ readily imply the relative closeness of $\lamTf$ to $\lambda_i$, i.e.
\begin{equation*}
  \min\limits_j \left| \frac{\lambda_j - \lamTf}{\lambda_j} \right| = \frac{\lamTf - \lambda_i}{\lambda_i}.
\end{equation*}
Using this fact, Theorem~\ref{th:estwb}, inequality \eqref{eq:absineq}, Lemma~\ref{le:normw}, and the bound $\underline{\lambda}_1 \leq \lambda_1$, we obtain
$$
  \frac{\lamTf - \lambda_i}{\lambda_i}
    = \min\limits_j \left| \frac{\lambda_j - \lamTf}{\lambda_j} \right|
    \leq | w |_b
    \leq \lambda_1^{-1/2} \| w \|_a
    \leq \underline{\lambda}_1^{-1/2} \eta.
$$
We finish the proof by observing that this inequality is equivalent to \eqref{eq:underlinelambdai}.
\end{proof}


\section{Local efficiency of the error indicators}
\label{se:loceff}

In this section we still keep $\lamTf\in\R$ and $\uTf\in\VT$ to be a fixed approximate eigenpair.
We proof that error indicators \eqref{eq:etaK} are bounded from above by a constant multiple of the classical residual error indicators \eqref{eq:etaRK}. This result is well known for various source problems, see e.g.~\cite{El_Al_Ern_Voh_a_post_NL_10,Ern_Voh_adpt_IN_13} and we provide its generalization to eigenvalue problems.
The technique of the proof is based on the nonconforming projection method of \cite{Arbo_Chen_implem_MM_nonc_met_sec_ord_el_prob_95,Arn_Bre_mix_NCFEMs_implem_postproc_EEs_85}, see also~\cite[Section~4.4.2]{Voh_un_apr_apost_MFE_10}.
The nonconforming projection methods enables to express the solution of the mixed problem \eqref{min_prob_qaO} alternatively as suitable projections of the solution of a specially constructed nonconforming finite element problem.

We start with the definition of the nonconforming finite element space $M_p(\Ta)$ on the patch $\Ta$. It consists of piecewise polynomial and in general discontinuous functions satisfying certain jump conditions on edges.
We first define the local space $M_p(K)$ as
\begin{align*}
M_p(K) &= \left\{ \begin{array}{lll}
  \{v_h \in P_{p+3}(K): v|_{\Gamma} \in P_{p+1}(\Gamma) \quad \Gamma \subset \pt K\}, \quad \text{if } p \text{ is even}, \\
  \{v_h \in P_{p+3}(K): v|_{\Gamma} \in P_{p}(\Gamma) \oplus \widetilde{P}_{p+2}(\Gamma) \quad \Gamma \subset \pt K\}, \quad \text{if } p \text{ is odd},
\end{array} \right.
\end{align*}
where $\widetilde{P}_{p+2}(\Gamma)$ denotes the one dimensional space generated by the Legendre polynomial of degree $p+2$ on the edge $\Gamma$.
Now, we define $M_p(\Ta)$ as the space of functions $m_h$ defined in $\Oma$ satisfying the following four conditions:
\begin{alignat}{2}
  \label{V0}
  m_h|_K &\in M_p(K) &\quad &\forall K\in\Ta,
  \\
  \label{V1}
  ([m_h], w_h )_\Gamma &= 0 &\quad &\forall w_h \in P_p(\Gamma),
      \ \forall \Gamma \in \cEaI,
  \\
  \label{V2}
  (m_h, w_h)_\Gamma &= 0 &\quad &\forall w_h \in P_p(\Gamma),
      \ \forall \Gamma \in \cEaBD, 
  \\
  \label{V3}
  (m_h, 1)_{\Oma} &= 0 &\quad &\forall \ta \in \NTI \cup \NTN.
\end{alignat}

In the nonconforming projection method, the mixed problem \eqref{min_prob_qaO} corresponds to the problem 
of finding 
$\tilde{\da} \in M_p(\Ta)$ such that
\begin{equation}\label{proj_method}
  \sum_{K \in \Ta} (\PiWK(\cA \nabla \tilde{\da}|_{K} - \HF \cA \nabla \uTf + \ttqa), \nabla \xi)_K
  = - (\rTa + \ddiv \ttqa, \Pi_p(\xi))_{\Oma}
\end{equation}
for all $\xi \in M_p(\Ta)$,
where $\PiWK$ is the $L^2(K)$-orthogonal projection with respect to the scalar product $(\cA^{-1} {\cdot}, {\cdot})_K$ onto the Raviart--Thomas space $\RT(K)$
and $\ttqa \in \bWa$ is fixed as in \eqref{min_prob_qaO}.

In \cite{Arbo_Chen_implem_MM_nonc_met_sec_ord_el_prob_95} it is shown that this problem
is uniquely solvable.
In addition, the result \cite[Theorem 1]{Arbo_Chen_implem_MM_nonc_met_sec_ord_el_prob_95} considered in the setting of problem \eqref{proj_method}, implies the following lemma.
\begin{lemma}
Let $\tilde{\da} \in M_p(\Ta)$ be the solution of problem \eqref{proj_method}.
Then functions
\begin{align}
  {\tqaO}|_{K} &= - \PiWK(\cA \nabla \tilde{\da} - \HF \cA \nabla \uTf + \ttqa)|_{K} \qquad \forall K \in \Ta, \label{postproc_1}\\
  {\daO}|_{K} &= \Pi_p(\tilde{\da}|_{K}) \qquad \forall K \in \Ta \label{postproc_2}
\end{align}
solve problem \eqref{min_prob_qaO}.
\end{lemma}

We now formulate and prove several auxiliary results.
\begin{lemma}
  Let $\tqa \in \bWa$, $m_h \in M_p(\Ta)$, and $\uTf \in \VT$. Then
   \begin{equation}\label{jump_qa}
      \sum_{K \in \Ta} \left(\tqa {\cdot} \tn_K, m_h \right)_{\pt K} =
      \sum_{\Gamma \in \cEaBN} (\Pi_\Gamma(\gTa), m_h)_\Gamma
   \end{equation}
and
   \begin{equation}\label{jump_grad_uB}
    \sum_{K \in \Ta} (\HF \cA \nabla \uTf {\cdot} \tn_K, m_h)_{\pt K} =
    \sum_{\Gamma \in \cEaI} \left( [\HF \cA \nabla \uTf]_\Gamma {\cdot} \tn_{\Gamma}, m_h \right)_{\Gamma} + \sum_{\Gamma \in \cEaBN} (\HF \cA \nabla \uTf \cdot \tn, m_h)_\Gamma,
   \end{equation}
where $m_h$ on an edge $\Gamma\in \cEaI$ can attain values from any of the two elements sharing $\Gamma$.
Note that sums over empty sets are considered as zero.
\end{lemma}
\begin{proof}
  First, we notice the identity
  $$
  \sum_{K \in \Ta} \left(\tqa {\cdot} \tn_K, m_h \right)_{\pt K}
  =  \sum_{\Gamma \in \cEaI} \left(\tqa {\cdot} \tn_\Gamma, [m_h]_\Gamma \right)_{\Gamma}
    +\sum_{\Gamma \in \cEaB} \left(\tqa {\cdot} \tn, m_h \right)_{\Gamma},
  $$
  where $\cEaB = \cEaBE \cup \cEaBD \cup \cEaBN$ is the set of edges on the boundary $\partial\Oma$.
  Since $\tqa|_\Gamma {\cdot} \tn_\Gamma \in P_p(\Gamma)$, we can use \eqref{V1} for all $\Gamma\in \cEaI$
  and, hence, the first sum on the right-hand side vanishes.
  Concerning the boundary edges $\Gamma \in \cEaB$, we have $\tqa {\cdot} \tn_\Gamma = 0$ for all $\Gamma\in \cEaBE$ by \eqref{eq:Wa}.
  On edges $\Gamma \in \cEaBD$ we have $\tqa {\cdot} \tn_\Gamma = 0$ by \eqref{V2}.
  Finally, on edges $\Gamma \in \cEaBN$ we have $\tqa {\cdot} \tn_\Gamma = \Pi_\Gamma(\gTa)$ by \eqref{eq:Wa} again.
  Consequently, \eqref{jump_qa} holds true.

  To prove \eqref{jump_grad_uB}, we consider an edge $\Gamma \in \cEaI$, its normal vector $\tn_\Gamma$, and elements $K^+$ and $K^-$ sharing this edge. Since $\HF \cA \nabla \uTf|_{K^-} \cdot \tn_{K^-} \in P_p(\Gamma)$, we can use \eqref{V1} to find that
  $$
    \left( \HF \cA \nabla \uTf|_{K^-} \cdot \tn_{K^-}, m_h|_{K^-} \right)_\Gamma = \left( \HF \cA \nabla \uTf|_{K^-} \cdot \tn_{K^-}, m_h|_{K^+} \right)_\Gamma
  $$
  and consequently
  $$
      \left( \HF \cA \nabla \uTf|_{K^+} \cdot \tn_{K^+}, m_h|_{K^+} \right)_\Gamma
    + \left( \HF \cA \nabla \uTf|_{K^-} \cdot \tn_{K^-}, m_h|_{K^-} \right)_\Gamma
    = \left( [\HF \cA \nabla \uTf]_\Gamma \cdot \tn_\Gamma, m_h \right)_\Gamma,
  $$
  where $m_h$ can be any of the two values $m_h|_{K^+}$ and $m_h|_{K^-}$.
  Thus, since $\HF$ vanishes on edges $\Gamma \in \cEaBE$ and $(\HF \cA \nabla \uTf \cdot \tn, m_h)_\Gamma = 0$ on edges $\Gamma\in\cEaBD$ by \eqref{V2},
  we obtain \eqref{jump_grad_uB}.
\end{proof}

We note that the set $\cEaBN$ is empty in many cases. For example if $\ta \in \NTI$ is an interior vertex then the right-hand side of \eqref{jump_qa} vanishes.

\begin{lemma} \label{le:FriedPoin_mh}
Let $h_\ta = \max_{K \in \Ta} h_K$, where
$h_K = \operatorname{diam}(K)$ denotes the diameter of the element $K \in \Ta$.
Then there exists a constant $C > 0$ such that
\begin{equation}\label{FriedPoin_mh}
  \norm{m_h}_{\Oma}^2 \leq C h_\ta^2 \sum_{K \in \Ta} \norm{\nabla m_h}_K^2
\end{equation}
holds for all $m_h \in M_p(\Ta)$.
\end{lemma}
\begin{proof}
   We denote by $H^1(\Ta) = \{ v \in L^2(\Oma): v|_K \in H^1(K) \ \forall K \in \Oma \}$
   the broken Sobolev space and consider the discrete nonconforming
   Poincar\'e inequality
   \begin{equation}\label{discr_Poinc}
       \|v_h\|_{\Oma}^2 \leq \CP \sum_{K \in \Ta} \|\nabla v_h\|_{K}^2 + \frac{4}{|\Oma|} (v_h, 1)_{\Oma}^2
   \end{equation}
   for all $v_h \in \{ z_h \in H^1(\Ta): ([z_h]_\Gamma,1)_{\Gamma} = 0 \quad \forall \Gamma \in \cEaI\}$.
   This inequality is given in \cite[Thm. 8.1]{Voh_on_discr_P-F_ineq_nonc_approx_H1_05}
   together with an explicit expression for $\CP$. This expression and the shape regularity \eqref{shapereg}
   then yield $\CP \leq C h_\ta^2$, where $C$ is independent of $h_\ta$.
      
   Similarly, the discrete nonconforming Friedrichs' inequality
   \begin{equation}\label{discr_Fried}
       \|v_h\|_{\Oma}^2 \leq \CFD \sum_{K \in \Ta} \|\nabla v_h\|_{K}^2 \quad
   \end{equation}
   holds for all $v_h \in \{ z_h \in H^1(\Ta): ([z_h]_\Gamma,1)_{\Gamma} = 0 \quad \forall \Gamma \in \cEaI \cup \cEaBD \}$,
   see \cite{Voh_on_discr_P-F_ineq_nonc_approx_H1_05}, and
   the constant $\CFD$ again satisfies $\CFD \leq C h_\ta^2$.

   Now, consider $\ta \in \NTI \cup \NTN$.
   In this case we derive \eqref{FriedPoin_mh} by applying \eqref{discr_Poinc}, which we can use because of \eqref{V1} and \eqref{V3}. 
   Similarly, if $\ta \in \NTD$ then we obtain \eqref{FriedPoin_mh} by \eqref{discr_Fried}, because $m_h$ satisfies \eqref{V1} and \eqref{V2}.
%
\end{proof}

Finally, we are in the position to prove that the error indicator \eqref{eq:etaK} is bounded from above by a constant multiple of the classical residual indicators \eqref{eq:etaRK}.
For that purpose we recall definitions \eqref{eq:omK} and \eqref{eq:TomK} of the patch of elements $\omK$ and the corresponding set $\cT(\omK)$ of elements contained in this patch.

\begin{theorem}\label{th:etaKetaRK}
   Let $\cF$ be a shape regular family of triangulations, see \eqref{shapereg}.
   Let the error indicator $\eta_K$ be given by \eqref{eq:etaK} with the flux reconstruction \eqref{eq:defq}.
   Let the classical residual error indicator $\eta_{R,K}$ be given by \eqref{eq:etaRK}.
   Then there exists a constant $C>0$ such that
   \begin{equation} \label{class_ind_vs_new_indicator}
     \eta_K^2 \leq C \sum_{K' \in \cT(\omK)} \eta_{R,K'}^2
     \quad\forall K\in\cT \text{ and } \forall \cT \in \cF.
   \end{equation}
\end{theorem}

\begin{proof}
  Let $\cT \in \cF$ be fixed. Given $K\in\cT$,
  we straightforwardly estimate $\eta_K$ as
  \begin{multline}
    \label{eq:etaKsuma}
    \eta_K 
       = \left\| \cA^{\frac12} \nabla \uTf - \cA^{-\frac12} \tqT\right\|_{K}
       = \left\|\sum_{\ta \in \NTK} \HF \cA^{\frac{1}{2}} \nabla \uTf - \cA^{-\frac{1}{2}} \tqa \right\|_K
       \\
       \leq \sum_{\ta \in \NTK} \left\|\HF \cA^{\frac{1}{2}} \nabla \uTf - \cA^{-\frac{1}{2}} \tqa \right\|_K
       \leq \sum_{\ta \in \NTK} \left\|\HF \cA^{\frac{1}{2}} \nabla \uTf - \cA^{-\frac{1}{2}} \tqa \right\|_{\Oma}.
  \end{multline}
  Thus, we fix a vertex $\ta \in \NT$ and estimate $\left\|\HF \cA^{\frac{1}{2}} \nabla \uTf - \cA^{-\frac{1}{2}} \tqa \right\|_{\Oma}$ as in \cite{Ern_Voh_adpt_IN_13}.

   In order to do that, we consider $\tilde{\da} \in M_p(\Ta)$ defined in \eqref{proj_method}
   and apply \cite[Lemma 5.4, statement (5.3)]{Voh_un_apr_apost_MFE_10}.
   We obtain
   \begin{equation}\label{equivalence}
      \| \PiWTa (-\cA \nabla \tilde{\da}) \|_{\Oma} \geq C 
      \| \cA^{\frac{1}{2}} \nabla \tilde{\da}\|_{\Oma},
   \end{equation}
   where the projection
   $\PiWTa$ is defined piecewise as $(\PiWTa \bw)|_K = \PiWK \bw$ for all $K \in \Ta$ and $\bw \in [L^2(\Oma)]^2$.
   Since $\cA^{-1}(\HF \cA \nabla \uTf - \tqa) \in \RT(K)$ for all $K \in \Ta$, we have
   \begin{equation}\label{projRTNp}
     \left( \cA^{-1}\left( \HF \cA \nabla \uTf - \tqa \right), \PiWTa (\cA \nabla \tilde{\da}) \right)_{\Oma} = \left( \HF \cA \nabla \uTf - \tqa, \nabla \tilde{\da} \right)_{\Oma}.
   \end{equation}
   Using the fact that \eqref{postproc_1} is equivalent to
   $
   \PiWK(\cA \nabla \tilde{\da}) = \HF \cA \nabla \uTf|_K - \tqa|_K
   $
   together with \eqref{projRTNp}, uniform positive definiteness of $\cA$,
   and \eqref{equivalence}, we obtain
   \begin{multline} \label{express_flux_est}
     \left\|\HF \cA^{\frac{1}{2}} \nabla \uTf - \cA^{-\frac{1}{2}} \tqa \right\|_{\Oma}
     = \left(\cA^{-1}\left( \HF \cA \nabla \uTf - \tqa \right),
              \frac{\PiWTa (\cA \nabla \tilde{\da})}
                   {\| \cA^{-\frac{1}{2}} \PiWTa (\cA \nabla \tilde{\da}) \|_{\Oma}}\right)_{\Oma}
   \\
     \leq C \left(\HF \cA \nabla \uTf - \tqa, \frac{\nabla \tilde{\da}}{\| \nabla \tilde{\da}\|_{\Oma}}\right)_{\Oma}
     \leq C \sup_{\substack{m_h \in M_p(\Ta) \\ \| \nabla m_h \|_{\Oma} = 1}} \left( \HF \cA \nabla \uTf - \tqa, \nabla m_h \right)_{\Oma}.
   \end{multline}

   Considering any $m_h \in M_p(\Ta)$, we apply the Green theorem,
   \eqref{eq:testPp},   
   \eqref{jump_qa}, \eqref{jump_grad_uB}, and derive equality
   \begin{multline*} 
    \left( \HF \cA \nabla \uTf - \tqa, \nabla m_h \right)_{\Oma}
    = \sum_{K \in \Ta} \left\{ -(\ddiv(\HF \cA \nabla \uTf - \tqa), m_h)_K \right.
    \\
    \left. + ((\HF \cA \nabla \uTf - \tqa) {\cdot} \tn_K, m_h)_{\pt K} \right\}
    = \sum_{K \in \Ta}
    \left( -\ddiv(\HF \cA \nabla \uTf) - \rTa, \Pi_p^K m_h \right)_K
    \\
    + \sum_{\Gamma \in \cEaI} \left( [\HF \cA \nabla \uTf]_\Gamma {\cdot} \tn_{\Gamma}, m_h \right)_{\Gamma}
    + \sum_{\Gamma \in \cEaBN} \left( \HF \cA \nabla \uTf {\cdot} \tn - \Pi_\Gamma(\gTa), m_h \right)_{\Gamma},
   \end{multline*}
   where $\Pi_p^K$ denotes the $L^2(K)$-projection onto $P_p(K)$.
   Using identities $-\ddiv(\HF \cA \nabla \uTf) - \rTa = \HF R$, see \eqref{eq:ra} and \eqref{eq:R},
   and $\HF \cA \nabla \uTf {\cdot} \tn|_{\Gamma} = \Pi_\Gamma( \HF \cA \nabla \uTf {\cdot} \tn|_{\Gamma})$
   together with definitions \eqref{eq:ga} and \eqref{eq:J}, we can express the above result in terms of the classical residual $R$ and jumps $J$ as
   \begin{multline} \label{express_flux_est_2}
    \left( \HF \cA \nabla \uTf - \tqa, \nabla m_h \right)_{\Oma}
    \\
    = \sum_{K \in \Ta}
    \left( \HF R, \Pi_p^K m_h \right)_K
    + \sum_{\Gamma \in \cEaI} \left( \HF J, m_h \right)_{\Gamma}
    + \sum_{\Gamma \in \cEaBN} \left( \Pi_\Gamma (\HF J), m_h \right)_{\Gamma}.
   \end{multline}

   Now, we recall the inverse trace inequality
   \begin{equation}
     \label{eq:invtrace}
     \|m_h \|_{\Gamma} \leq C h_{\Gamma}^{-\frac{1}{2}}  \|m_h \|_{K}
   \end{equation}
   which holds for any  polynomial $m_h$ on a triangle $K$ and its edge $\Gamma$
   (see 
   \cite{WarHes2003} for an explicit value of $C$).
   Using the Cauchy--Schwarz inequality,
   this inverse trace inequality,
   \eqref{FriedPoin_mh},
   quasi-uniformity of triangulations~\eqref{locquasiuni},
   and properties of projections $\Pi_p^K$ and $\Pi_\Gamma$,
   we estimate \eqref{express_flux_est_2} as follows:
   \begin{multline} \label{express_flux_est_3}
    \left( \HF \cA \nabla \uTf - \tqa, \nabla m_h \right)_{\Oma}
    \leq \left\{\sum_{K \in \Ta} h_K^{-2} \|m_h\|^2_K \right\}^\frac{1}{2}
           \left\{\sum_{K \in \Ta} h_K^{2} \| \HF R\|^2_K \right\}^\frac{1}{2}
     \\
     + \left\{\sum_{\Gamma \in \cEaI} h_{\Gamma}^{-1} \|m_h\|^2_{\Gamma} \right\}^\frac{1}{2} \left\{ \sum_{\Gamma \in \cEaI} h_{\Gamma} \| \HF J \|^2_{\Gamma} \right\}^\frac{1}{2}
     + \left\{\sum_{\Gamma \in \cEaBN} h_{\Gamma}^{-1} \|m_h\|^2_{\Gamma} \right\}^\frac{1}{2}
     \left\{ \sum_{\Gamma \in \cEaBN} h_{\Gamma} \| \Pi_\Gamma( \HF J ) \|^2_{\Gamma} \right\}^\frac{1}{2}
     \\
     \leq C \|\nabla m_h\|_{\Oma} \left\{ \sum_{K \in \Ta} h_K^{2} \|R\|^2_K
     +
     \sum_{\Gamma \in \cEaI} h_{\Gamma} \| J \|^2_{\Gamma}
     + \sum_{\Gamma \in \cEaBN} h_{\Gamma}
        \| J \|^2_{\Gamma}
     \right\}^\frac{1}{2}.
   \end{multline}

   Combining \eqref{express_flux_est_3} and \eqref{express_flux_est} finally yields
   \begin{multline} \label{express_flux_est_fin}
    \left\|\HF \cA^{\frac{1}{2}} \nabla \uTf - \cA^{-\frac{1}{2}} \tqa \right\|_{\Oma}
     \leq C \left\{ \sum_{K \in \Ta} h_K^{2} \|R\|^2_K
     + \sum_{\Gamma \in \cEaI \cup \cEaBN} h_{\Gamma} \| J \|^2_{\Gamma}
     \right\}^\frac{1}{2}.
   \end{multline}
   Using this estimate in \eqref{eq:etaKsuma} finishes the proof.

\end{proof}

Theorem~\ref{th:etaKetaRK} yields immediately the local efficiency of the error indicators $\eta_K$ provided the classical residual estimator $\eta_{R,K}$ is locally efficient.
Combining this result with Lemma~\ref{le:effclasres}, we obtain the following corollary,
where we use a wider patch of elements defined as
\begin{equation}
  \label{eq:tomK}
  \tomK = \interior\bigcup \{ K' \in \cT : K' \cap \omK \neq \emptyset \}.
\end{equation}


\begin{corollary}[Efficiency of indicators]
Let $w$ be given by \eqref{eq:defw} 
with bilinear forms \eqref{eq:blf}, \eqref{eq:blfb}, and with $\lambda_* = \lamTf$ and $u_* = \uTf$.
Then the error indicators given by \eqref{eq:etaK} satisfy
\begin{equation}
  \label{eq:effetaK}
  \eta_K \leq C \norm{w}_{a,\tomK}.
\end{equation}
\end{corollary}
\begin{proof}
This is an immediate consequence of \eqref{class_ind_vs_new_indicator} and \eqref{eq:effetaRK}.
\end{proof}

\section{Adaptive algorithm}
\label{se:adaptalg}

In this section we consider a general adaptive algorithm, list conditions guaranteeing its convergence, and verify that an algorithm based on error indicators \eqref{eq:etaK} satisfies these conditions and is convergent.

We consider the following standard adaptive loop:
\begin{itemize}
\item[1.] $(\lamTkf,\uTkf) = \SOLVE(\cT_k)$
\item[2.] $\{\eta_K^{\cT_k}\} = \ESTIMATE(\cT_k, \lamTkf, \uTkf)$
\item[3.] $\cM_k = \MARK(\cT_k, \{\eta_K^{\cT_k}\} )$
\item[4.] $\cT_{k+1} = \REFINE(\cT_k, \cM_k)$
\end{itemize}

We start this loop with an initial mesh $\cT_0$. Module \SOLVE\ returns an approximate solution on the actual mesh $\cT_k$. Module \ESTIMATE\ computes the error indicators $\eta_K^{\cT_k}$ for all $K \in \cT_k$.
Module \MARK\ determines the set of elements $\cM_k$, which are subsequently refined by the module \REFINE\ and a new mesh $\cT_{k+1}$ is constructed.

We will use the result \cite{GarMor2011} to show the convergence of this algorithm for eigenvalue problems.
Before we list the required assumptions on respective modules of the adaptive algorithm,
we introduce certain notions.

Let $\cT_0$ be an arbitrary but fixed initial mesh. Let $\cF(\cT_0)$ stands for the family of all meshes that can be produced by a finite number of successive applications of \REFINE\ with all possible sets of marked elements. Given a function $u\in V$, we define the residual $\bR(u): V \rightarrow \R$ corresponding to the eigenproblem \eqref{eq:eigp} as the linear and continuous functional on $V$ given by
\begin{equation}
  \label{eq:residual}
  \langle \bR(u),v \rangle = a(u,v) - \lambda b(u,v) \quad \forall v \in V,
\end{equation}
where $\lambda = a(u,u)/b(u,u)$.
Considering a mesh $\cT \in \cF(\cT_0)$ and an approximate eigenpair $\lamTf \in \R$, $\uTf \in \VT$ given by \eqref{eq:discreigp}, we say that error indicators $\eta_K$ provide an upper bound for the residual \cite{GarMor2011} if
there exists a constant $C>0$ (uniform over the family $\cF(\cT_0)$) such that
\begin{equation}
  \label{eq:ubr}
  | \langle \bR(\uTf),v \rangle | \leq C \sum_{K\in\cT} \eta_K \| \nabla v \|_{\omK}
  \quad\forall v \in V.
\end{equation}
Similarly, indicators $\eta_K$ are said to be stable with respect to $\uTf$ \cite{GarMor2011} if
there exists a constant $C>0$ (uniform over the family $\cF(\cT_0)$) such that
\begin{equation}
  \label{eq:stability}
  \eta_K \leq C \| \uTf \|_{H^1(\omK)}.
\end{equation}
Now, we are ready to list the assumptions on the respective modules of the adaptive algorithm.

(\Aslv) Given a mesh $\cT \in \cF(\cT_0)$, module \SOLVE\ provides an approximate eigenpair $\lamTf \in \R$ and $\uTf \in \VT$ satisfying \eqref{eq:discreigp}, where $\VT$ is given by \eqref{eq:VT}. Note that we require conformity, i.e. $\VT \subset H^1(\Omega)$. Consequently, if $\cTstar$ is a refinement of $\cT$ then $\VT \subset \VTstar$ and the minimum-maximum principle implies $\lambda_i \leq \lambda^{\cT_*}_i \leq \lamT_i$ for all $i=1,2,\dots,\operatorname{dim} \VT$.

(\Aest) Given a mesh $\cT \in \cF(\cT_0)$ and $\lamTf \in \R$, $\uTf \in \VT$ produced by \SOLVE, module \ESTIMATE\ produces error indicators $\eta^\cT_K$ that provide an upper bound for the residual \eqref{eq:ubr} and that are stable with respect to $\uTf$ in the sense of \eqref{eq:stability}.

(\Amrk) Module \MARK\ is assumed to be reasonable \cite{GarMor2011}. This means that
the set of elements $\cM_k$ marked for the refinement contains at least one element $K^\mathrm{max}$ such that $\eta_{K^\mathrm{max}} = \max_{K \in \cT_k} \eta_K$.

(\Aref) Module \REFINE\ has to refine all marked elements at least once.
Module \REFINE\ is assumed to produce shape regular families of triangulations, i.e.,
the family $\cF(\cT_0)$ is shape regular in the sense of \eqref{shapereg}.
Further, module \REFINE\ has to satisfy the assumption of the unique quasi-regular element subdivision of \cite{Mor_Sieb_Vees_A_basic_res_conf_adap_FEs_08}.
This means that there exit constants $q_1,q_2 \in (0,1)$ such that whenever an element $K$ is refined by
\REFINE\ into $n(K)$ subelements $K'_1$, $K'_2$, \dots, $K'_{n(K)}$ then
\begin{equation}
  q_1 |K| \leq |K'_i| \leq q_2 |K| \quad\forall i=1,2,\dots,n(K).
  \label{quasiregsplit}
\end{equation}
Note that by a refinement of an element we automatically mean that the resulting subelements form a partition of $K$, i.e.
$$
  K = K'_1 \cup K'_2 \cup \cdots \cup K'_{n(K)}
  \quad\text{and}\quad
  |K| = |K'_1| + |K'_2| + \cdots + |K'_{n(K)}|.
$$
Furthermore, due to the conformity, all meshes produced by \REFINE\ have to be compatible with $\GammaD$ and $\GammaN$.

These assumptions guarantee convergence of the above adaptive algorithm as it is shown in \cite{GarMor2011} and \cite{Mor_Sieb_Vees_A_basic_res_conf_adap_FEs_08}. We now verify that these assumptions are satisfied for the eigenvalue problem \eqref{eq:eigp} and for the error indicators \eqref{eq:etaK}.

Assumptions (\Aslv) on the module \SOLVE\ are satisfied by a standard implementation of the conforming finite element method of order $p$. Note that strictly speaking, we assume that there are no round-off errors and that the corresponding matrix eigenvalue problems are solved exactly.
Module \MARK\ satisfies the assumption (\Amrk) virtually always, because almost all existing marking strategies are reasonable. In numerical examples below, we implement the bulk criterion of \cite{Dorfler:1996}.
Concerning module \REFINE, the assumption of quasi-regular element subdivision is very natural and it is satisfied by all standard subdivision strategies. However, the crucial assumption of the shape regularity is often not easy to prove.
We consider the newest vertex bisection procedure to generate shape regular families of triangulations, see e.g. \cite{SchSie2005}.

Finally, we need to verify the assumptions on the module \ESTIMATE, namely the upper bound on the residual and the stability. The following two lemmas show that indicators \eqref{eq:etaK} satisfy these assumptions.
\begin{lemma}[Guaranteed upper bound for the residual]
Let $\cT \in \cF(\cT_0)$ be a triangulation of $\Omega$.
Let $\lamTf \in \R$, $\uTf \in \VT$ with $i\geq 1$ be an approximate eigenpair given by \eqref{eq:discreigp}. Let the residual $\bR(\uTf)$ be given by \eqref{eq:residual}, and the indicators $\eta_K$ and the estimator $\eta$ by \eqref{eq:etaK}.
Then indicators $\eta_K$ provide the upper bound on the residual \eqref{eq:ubr}.
Moreover,
\begin{equation} \label{eq:gubr}
  | \langle \bR(\uTf),v \rangle | \leq \sum_{K\in\cT} \eta_K \| \nabla v \|_{\cA,K}
  \quad\text{and}\quad
  \| \bR(\uTf) \|_{a'} = \| w \|_a,
\end{equation}
where $\|\cdot\|_{a'}$ is the dual norm corresponding to $V$ endowed with $\|\cdot\|_a$.
\end{lemma}
\begin{proof} The proof is a straightforward variant of Lemma~\ref{le:normw}:
\begin{equation} \label{eq:gubr2}
  | \langle \bR(\uTf),v \rangle | = | a(w,v) | = \left| \sum_{K\in\cT} (\cA \nabla \uTf - \tqT, \nabla v) \right|
  \leq \sum_{K\in\cT} \| \cA \nabla \uTf - \tqT \|_{\cA,K}  \| \nabla v \|_{\cA,K},
\end{equation}
which is the first estimate in \eqref{eq:gubr}.
The second statement in \eqref{eq:gubr} follows from the identity
$ \langle \bR(\uTf),v \rangle  = a(w,v)$ for all $v \in V$.
%
Finally, the upper bound on the residual \eqref{eq:ubr} follows from \eqref{eq:gubr2}, because
$\| \nabla v \|_{\cA,K} \leq \| \cA^\frac{1}{2} \|_{L^\infty(K)} \| \nabla v \|_K$.
\end{proof}

\begin{lemma}[Stability of classical residual indicators]\label{le:classtab}
Let $N_0 = \operatorname{dim} V^{\cT_0}$ be the number of degrees of freedom corresponding to the initial mesh $\cT_0$.
Let $\cT \in \cF(\cT_0)$ be a mesh and let $\lamTf \in \R$, $\uTf \in \VT$ with $1\leq i \leq N_0$ be an approximate eigenpair given by \eqref{eq:discreigp}.
Let $K\in \cT$ be fixed.
Let the classical residual indicators $\eta_{R,K}$ be given by \eqref{eq:etaRK}.
Then there exists $C>0$, which depends on problem data and on the initial mesh, but is independent of $h_K$, such that
$$
  \eta_{R,K} \leq C \| \uTf \|_{H^1(\omega_K)}.
$$
\end{lemma}
\begin{proof}
Let $K \in\cT$ be fixed.
Since $\cA$ is constant on $K$ and $\uTf$ is a polynomial on $K$, we can use the inverse inequality and
derive the estimate
$$
  \| \ddiv(\cA \nabla \uTf) \|_K
  \leq \| \cA \|_{L^\infty(\Omega)} \| \Delta \uTf \|_K
  \leq C h_K^{-1} \| \nabla \uTf \|_K.
$$
This can be used to bound $h_K \| R \|_K$, see \eqref{eq:R}, as follows
\begin{multline}
\label{eq:Rest}
  h_K \| R \|_K
    \leq h_K \| \lamTf \beta_1 \uTf - c \uTf + \ddiv(\cA \nabla \uTf) \|_K
\\
    \leq h_K | \lamTf \beta_{1|K} - c_{|K} | \| \uTf \|_K + h_K \| \ddiv(\cA \nabla \uTf) \|_K
    \leq C \| \uTf \|_{H^1(K)},
\end{multline}
where we use the fact that $\lamTf$ is bounded from above by $\lambda^{\cT_0}_i$, which is determined on the initial mesh $\cT_0$ and is included in $C$.
To bound $\| J \|_{\partial K}$, see \eqref{eq:J}, we consider an edge $\Gamma \subset \partial K$ and distinguish three cases. If $\Gamma$ is an interior edge, then there exists an elements $K'$ such that $\Gamma = K \cap K'$ and
\begin{multline}
\label{eq:JestI}
  h_\Gamma^{1/2} \| J \|_\Gamma
    = h_\Gamma^{1/2} \| (\cA \nabla \uTf)|_K \cdot \tn_{K} + (\cA \nabla \uTf)|_{K'}\cdot \tn_{K'} \|_\Gamma
\\
    \leq C \left( \| \cA \nabla \uTf \|_K + \| \cA \nabla \uTf \|_{K'} \right)
    \leq C \| \cA \|_{L^\infty(\Omega)} \| \uTf \|_{H^1(K \cup K')},
\end{multline}
where we use the inverse trace inequality \eqref{eq:invtrace}.
Similarly, if $\Gamma \subset \GammaN$ then we use the same inverse trace inequality to derive estimate
\begin{multline}
\label{eq:JestN}
  h_\Gamma^{1/2} \| J \|_\Gamma
    = h_\Gamma^{1/2} \| \cA \nabla \uTf \cdot \tn_K + \alpha \uTf - \lamTf\beta_2 \uTf \|_\Gamma
\\
    \leq C \left( \| \cA \nabla \uTf \|_K + | \alpha_{|K} - \lamTf \beta_{2|K} | \| \uTf \|_K \right)
    \leq C \| \uTf \|_{H^1(K)},
\end{multline}
where we again bound $\lamTf$ by $\lambda^{\cT_0}_i$.
Finally, if $\Gamma \subset \GammaD$ then $J = 0$ and we finish the proof by combining \eqref{eq:Rest}, \eqref{eq:JestI}, and \eqref{eq:JestN}.
\end{proof}

\begin{lemma}[Stability of error indicators]\label{le:stability}
Let $\cT \in \cF(\cT_0)$ and $N_0 = \operatorname{dim} V^{\cT_0}$.
Let $\lamTf \in \R$ and $\uTf \in \VT$ with $1 \leq i \leq N_0$ be an approximate eigenpair given by \eqref{eq:discreigp} and
let $|\uTf|_b = 1$. Then
\begin{equation}
\label{eq:stabind}
  \eta_K \leq C \| \uTf \|_{H^1(\tomK)} \quad \forall K \in\cT
  \quad\text{and}\quad
  \eta_K \leq C_\eta,
\end{equation}
where $C$ and $C_\eta$ are uniform constants over the family of meshes
and $\tomK$ is given in \eqref{eq:tomK}.
\end{lemma}
\begin{proof}
The first statement follows immediately from Theorem~\ref{th:etaKetaRK} and Lemma~\ref{le:classtab}.
The second statment follows form the equivalence of the energy and $H^1$ norm, the fact that $\| \uTf \|_a^2 = \lamTf$, and from the bound $\lamTf \leq \lambda_i^{\cT_0}$:
$$
\| \uTf \|_{H^1(\tomK)} \leq C \| \uTf \|_a = C \sqrt{\lamTf} \leq C_\eta.
$$
\end{proof}

As in \cite{GarMor2011}, we define the set of normalized eigenfunctions corresponding to a given eigenvalue $\lambda \in \R$ as
$$
  \tilde E_\lambda = \{ u \in V :\quad a(u,v) = \lambda b(u,v) \quad \forall v \in V
  \quad\text{and}\quad
  | u |_b = 1 \}.
$$
This enables to formulate the analogous convergence theorem as \cite[Theorem~3.10]{GarMor2011}.
\begin{theorem}
\label{th:conv}
Let the adaptive algorithm be driven by error indicators \eqref{eq:etaK}, let it satisfies assumptions (\Aslv), (\Aest), (\Amrk), and (\Aref), and
let $\{ \lamTkf \}$ and $\{ \uTkf \}$ be the generated sequence of approximate eigenvalues and eigenfunctions for a fixed $i \in \{1,2,\dots, N_0\}$, $N_0 = \operatorname{dim} V^{\cT_0}$.
Then $\{ \lamTkf \}$ is non-increasing and
there exists an eigenvalue $\lambda \in \R$ such that
%
\begin{equation}
  \label{eq:convergence}
  \lim_{k\rightarrow\infty} \lamTkf = \lambda
  \quad\text{and}\quad
  \lim_{k\rightarrow\infty} \operatorname{dist}_{H^1(\Omega)} (\uTkf, \tilde E_\lambda) = 0.
\end{equation}
\end{theorem}
\begin{proof}
The only difference from \cite[Theorem~3.10]{GarMor2011} is in the stability of error indicators. The stability in \cite{GarMor2011} is considered with respect to the patch $\omK$, while in \eqref{eq:stabind} we naturally obtained the stability with respect to the wider patch $\tomK$. However, this makes no difference, because we still have $|\tomK| \leq C h^d_{K}$ due to the uniform boundedness of the number of elements in $\tomK$.
\end{proof}

A problem of Theorem~\ref{th:conv} is that $\lambda$ need not be necessarily $\lambda_i$ as we would expected.
In some quite pathological cases $\lambda$ can be an eigenvalue greater than $\lambda_i$. We will follow \cite{GarMor2011} and present \cite[Theorem~3.12]{GarMor2011} showing the convergence towards $\lambda_i$.
This result is based on the nondegeneracy assumption.
Problem \eqref{eq:eigp} satisfies the nondegeneracy assumption if all eigenfunctions $u_i \in V$ satisfying \eqref{eq:eigp} are such that $u_i|_{\mathcal{O}} \not\in P_p(\mathcal{O})$ for all nonempty open subsets $\mathcal{O}$ of $\Omega$, see \cite{GarMor2011}.
\begin{theorem}
\label{th:convi}
Let problem \eqref{eq:eigp} satisfy the nondegeneracy assumption. Then under the conditions of Theorem~\ref{th:conv} we have
\begin{equation}
  \label{eq:convergencei}
  \lim_{k\rightarrow\infty} \lamTkf = \lambda_i
  \quad\text{and}\quad
  \lim_{k\rightarrow\infty} \operatorname{dist}_{H^1(\Omega)} (\uTkf, \tilde E_{\lambda_i}) = 0.
\end{equation}
\end{theorem}
\begin{proof}
See \cite[Theorem~3.12]{GarMor2011}.
\end{proof}

Finally, we are in the position to prove the convergence of ${\underline{\lambda}}_i^{\cT_k}$ and to show that
this value is a guaranteed lower bound on the corresponding exact eigenvalue $\lambda_i$ provided we have performed sufficiently many adaptive steps.
\begin{theorem}
\label{th:convlow}
If problem \eqref{eq:eigp} satisfies the nondegeneracy assumption and if conditions of Theorem~\ref{th:conv} hold, then ${\underline{\lambda}}_i^{\cT_k}$ defined by \eqref{eq:underlinelambda} and \eqref{eq:underlinelambdai} satisfy
\begin{equation}
  \label{eq:convergencelow}
  \lim_{k\rightarrow\infty} {\underline{\lambda}}_i^{\cT_k} = \lambda_i.
\end{equation}
Moreover, there exists $k_0 > 0$ such that
$$
  {\underline{\lambda}}_i^{\cT_k} \leq \lambda_i
  \quad \forall k \geq k_0.
$$
\end{theorem}
\begin{proof}
First, we prove the convergence of ${\underline{\lambda}}_i^{\cT_k}$ to $\lambda_i$.
Let $w^{\cT_k} \in V$ be defined by \eqref{eq:defw} with $\lambda_* = \lamTkf$ and $u_* = \uTkf$.
Due to \eqref{eq:convergencei}, we can easily show that $\| w^{\cT_k} \|_a \rightarrow 0$
and hence $w^{\cT_k} \rightarrow 0$ in $V$.
Consequently, the efficiency result \eqref{eq:effetaK} yields that $\eta^{\cT_k}_K$ tend to zero for all $K\in\cT_k$ and, thus, $\eta^{\cT_k}$ tends to zero as well.
Formulas \eqref{eq:underlinelambda} and \eqref{eq:underlinelambdai} together with \eqref{eq:convergencei}
now easily imply the convergence
${\underline{\lambda}}_i^{\cT_k} \rightarrow \lambda_i$ as $k\rightarrow\infty$.

Now, since $\lamTkf \rightarrow \lambda_i$, there exists $k_0 > 0$ such that the relative closeness assumption \eqref{eq:closestcond} for $i=1$ and \eqref{eq:closestcondi} for $i \geq 2$ is satisfied for all $k \geq k_0$.
Consequently, Theorems~\ref{th:lowerbound} and \ref{th:lowerboundi} guarantee that ${\underline{\lambda}}_i^{\cT_k} \leq \lambda_i$ for $k\geq k_0$.
\end{proof}

A distinctive feature of the error estimator $\eta$ given in \eqref{eq:etaK} is the possibility to use it in \eqref{eq:underlinelambda} or in \eqref{eq:underlinelambdai} and compute cheaply the lower  bound $\underline{\lambda}_i^\cT$ on $\lambda_i$. This is important for reliable stopping criteria.

For example, if the goal is to approximate the eigenvalue $\lambda_i$ within a prescribed relative error tolerance $\Ereltol$, then we stop the adaptive algorithm as soon as the estimate $\Erelest = \left(\lamTkf-\underline{\lambda}_i^{\cT_k} \right) / \underline{\lambda}_i^{\cT_k}$ of the true relative error $\Erel = (\lamTkf - \lambda_i) / \lambda_i$
is below the tolerance $\Ereltol$, i.e. as soon as
\begin{equation}
  \label{eq:relerr}
  \Erelest \leq \Ereltol.
\end{equation}
This inequality together with Theorem~\ref{th:convlow} yields $\Erel \leq \Erelest \leq \Ereltol$.
Thus, as soon as the algorithm succeeds to satisfy \eqref{eq:relerr}, then the true relative error $\Erel$ is really below the prescribed error tolerance.
Of course, this is guaranteed only if 
problem \eqref{eq:eigp} satisfies the nondegeneracy assumption and the number of adaptive steps $k$ is sufficiently large such that the relative closeness conditions \eqref{eq:closestcond} or \eqref{eq:closestcondi} hold.

The nondegeneracy assumption is not limiting in most cases and \cite[Lemma~3.13]{GarMor2011} recalls sufficient conditions for its validity.
A practical difficulty is the verification of the sufficient accuracy of the approximation $\lamTkf$
such that the relative closeness condition \eqref{eq:closestcond} or \eqref{eq:closestcondi} is satisfied.
Guaranteed verification of these conditions is not possible, unless guaranteed lower bounds of eigenvalues are known.
However, since the algorithm is proved to be convergent, we can have a good confidence that the relative closeness conditions hold as soon as the two-sided bounds of eigenvalues are resolved with sufficient accuracy. 
To be concrete, if $\underline{\lambda}_i^{\cT_\kfin}$ and $\underline{\lambda}_{i+1}^{\cT_\ellfin}$ denote the lower bounds computed in the final adaptive steps $\kfin$ and $\ellfin$, respectively, then we test the condition
\begin{equation}
  \label{eq:rctest}
  \lamTkf \leq 2 \left( \left(\underline{\lambda}_i^{\cT_{\kfin}}\right)^{-1} + \left(\underline{\lambda}_{i+1}^{\cT_{\ellfin}}\right)^{-1} \right)^{-1}
\end{equation}
for all previous adaptive steps $k=1,2,\dots,\kfin$. For those adaptive steps, where this test passes, 
we have a good confidence in the validity of the relative closeness condition.


\section{Numerical examples}
\label{se:numex}
We illustrate the numerical performance of the method by solving problem \eqref{eq:eigp_strong} 
in a dumbbell shaped domain 
$\Omega = (0,\pi)^2 \cup [\pi,4\pi/3]\times(\pi/3,2\pi/3) \cup (4\pi/3,7\pi/3)\times(0,\pi)$
with mixed boundary conditions. The chosen portions $\GammaD$ and $\GammaN$ of the boundary
$\partial\Omega$ are depicted in Figure~\ref{fi:dumbbell} (left).
Coefficients are constant with values $\beta_1 = \beta_2 = 1$, $c = \alpha = 0$, and 
$\cA$ being the identity matrix.

We compute the first ten eigenvalues of this problem by the standard linear finite element method,
i.e. we choose $p=1$ in \eqref{eq:VT}. Approximate eigenpairs are given 
by \eqref{eq:discreigp}. We use the adaptive algorithm described in Section~\ref{se:adaptalg}
and error indicators \eqref{eq:etaK} to steer it. The flux reconstruction is computed
on patches of elements by solving problems \eqref{min_prob_qa}. 
In every adaptive step, we compute
the lower bounds by using \eqref{eq:underlinelambda} for the first eigenvalue and 
\eqref{eq:underlinelambdai} for the subsequent eigenvalues.
We use the stopping criterion \eqref{eq:relerr} with $\Ereltol = 0.01$.

\begin{figure}
\tikzset{
  big arrowhead/.style={
    decoration={markings, mark=at position 1 with {\arrow[scale=1.5,thin,black]{>}} },
    postaction={decorate},
    shorten >=0.4pt}}
\begin{center}    
\begin{tikzpicture}[scale=0.86]
\draw [ultra thick] (0,0)--(3,0)--(3,1)--(4,1)--(4,0)--(7,0);
\draw [semithick] (7,0)--(7,3)--(4,3)--(4,2)--(3,2)--(3,3)--(0,3)--(0,0);
\node [below] at (5.5,0) {$\GammaD$};
\node [above] at (5.5,3) {$\GammaN$};
\draw [big arrowhead, thin] (0,2.5)--(3,2.5);
\draw [big arrowhead, thin] (3,2.5)--(0,2.5);
\node [below] at (1.5,2.5) {$\pi$};
\draw [big arrowhead, thin] (3,2.5)--(4,2.5);
\draw [big arrowhead, thin] (4,2.5)--(3,2.5);
\node [above] at (3.5,2.5) {$\displaystyle\frac{\pi}{3}$};
\draw [big arrowhead, thin] (4,2.5)--(7,2.5);
\draw [big arrowhead, thin] (7,2.5)--(4,2.5);
\node [below] at (5.5,2.5) {$\pi$};
\draw [big arrowhead, thin] (0.5,0)--(0.5,3);
\draw [big arrowhead, thin] (0.5,3)--(0.5,0);
\node [right] at (0.5,1.5) {$\pi$};
\draw [big arrowhead, thin] (3.5,1)--(3.5,2);
\draw [big arrowhead, thin] (3.5,2)--(3.5,1);
\node [right] at (3.5,1.5) {$\displaystyle\frac{\pi}{3}$};
\end{tikzpicture}
\quad
\raisebox{5mm}{\includegraphics[width=0.475\textwidth]{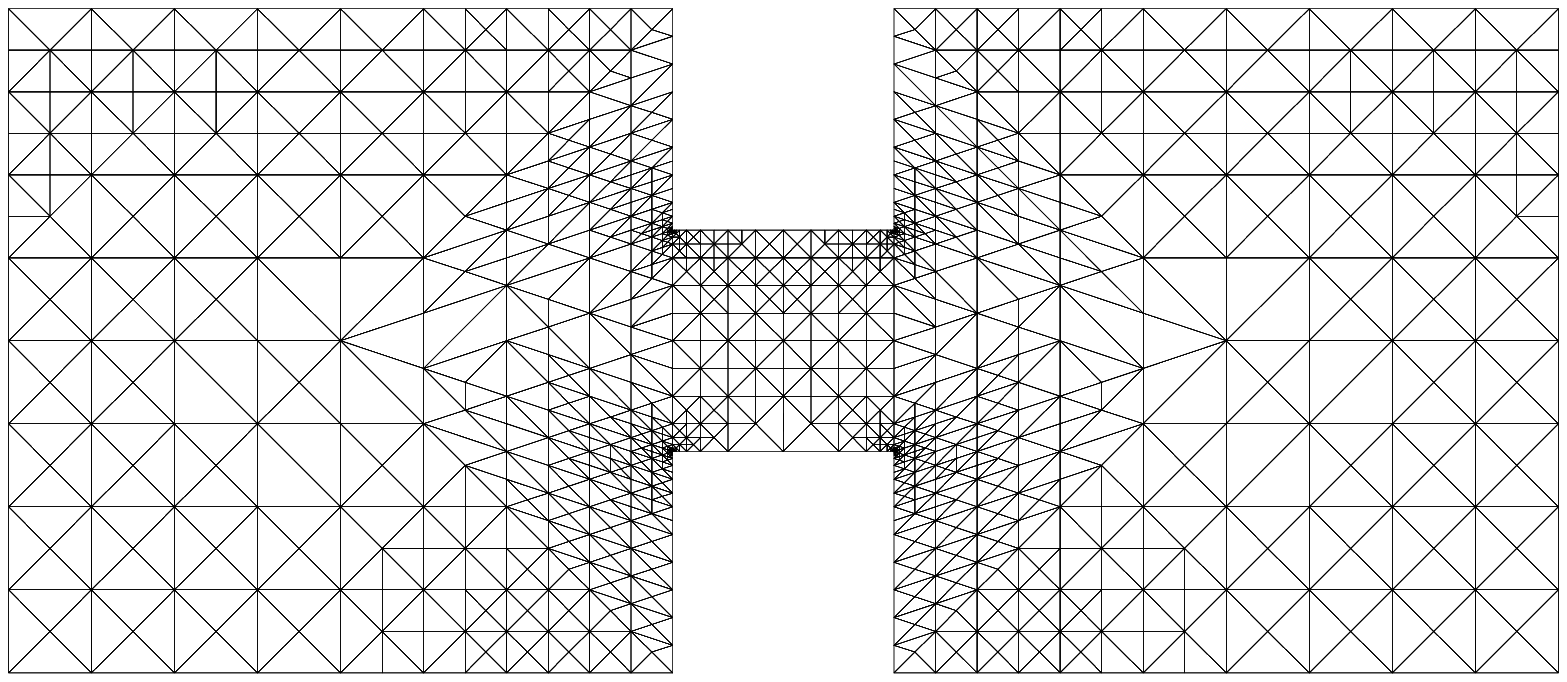}}
\end{center}
\caption{\label{fi:dumbbell}
The left panel shows the dimensions of the dumbbell shaped domain
and the Dirichlet (thick line) and Neumann (thin line) portions of its boundary.
The right panel presents the adaptively refined mesh in 
the 20th (out of 39) adaptive step of the computation of $\lambda_1$.}
\end{figure}

Table~\ref{ta:results} presents the lower and the upper bounds on eigenvalues
obtained in the final adaptive step. The corresponding numbers of degrees of freedom $N_\mathrm{DOF}$
and numbers of adaptive steps $N_\mathrm{AS}$ are included as well.
The table reveals pairs of closely clustered eigenvalues. 
Due to these tight clusters the eigenvalue problem in the dumbbell shaped domain is challenging to solve.
Moreover, most of the eigenfunctions have singularities at the re-entrant corners
of the domain.
The presented adaptive algorithm captures well theses singularities
and the meshes are automatically refined towards the re-entrant corners.
See Figure~\ref{fi:dumbbell} (right) for an illustration
of the adaptively refined mesh and Figure~\ref{fi:eigfun} for contour plots of the first two eigenfunctions.

\begin{table}
\begin{center}
\setlength{\tabcolsep}{4pt}
\begin{tabular}{c|cccc}
\hline
             &  lower  &  upper   & $N_\mathrm{DOF}$ & $N_\mathrm{AS}$ \\ 
\hline
   $\lambda_1$ &  0.1391 & 0.1405 &  20\,347 & 38 \\
   $\lambda_2$ &  0.1492 & 0.1507 &  24\,065 & 39 \\ 
   $\lambda_3$ &  0.4186 & 0.4226 & 101\,774 & 49 \\ 
   $\lambda_4$ &  0.4399 & 0.4443 & 137\,123 & 50 \\ 
   $\lambda_5$ &  0.8928 & 0.9011 & 343\,431 & 60 \\ 
\hline 
\end{tabular}
\quad
\begin{tabular}{c|cccc}
\hline
             &  lower  &  upper  & $N_\mathrm{DOF}$ & $N_\mathrm{AS}$ \\ 
\hline
   $\lambda_6$ & 0.8941  & 0.9025 &    318\,054 & 60 \\ 
   $\lambda_7$ & 1.1622  & 1.1731 &    562\,986 & 61 \\
   $\lambda_8$ & 1.1634  & 1.1745 &    575\,888 & 61 \\
   $\lambda_9$ & 1.2971  & 1.3100 &    809\,915 & 60 \\
$\lambda_{10}$ & 1.8212  & 1.8383 & 1\,180\,537 & 81 \\
\hline  
\end{tabular}
\end{center}
\caption{\label{ta:results}
The lower and upper bounds on the first ten eigenvalues
computed adaptively with the relative error tolerance $\Ereltol = 0.01$, see \eqref{eq:relerr}.
Columns $N_\mathrm{DOF}$ and $N_\mathrm{AS}$ 
present the final numbers of degrees of freedom and the numbers of adaptive steps performed.
}
\end{table}

\begin{figure}
\includegraphics[width=0.48\textwidth]{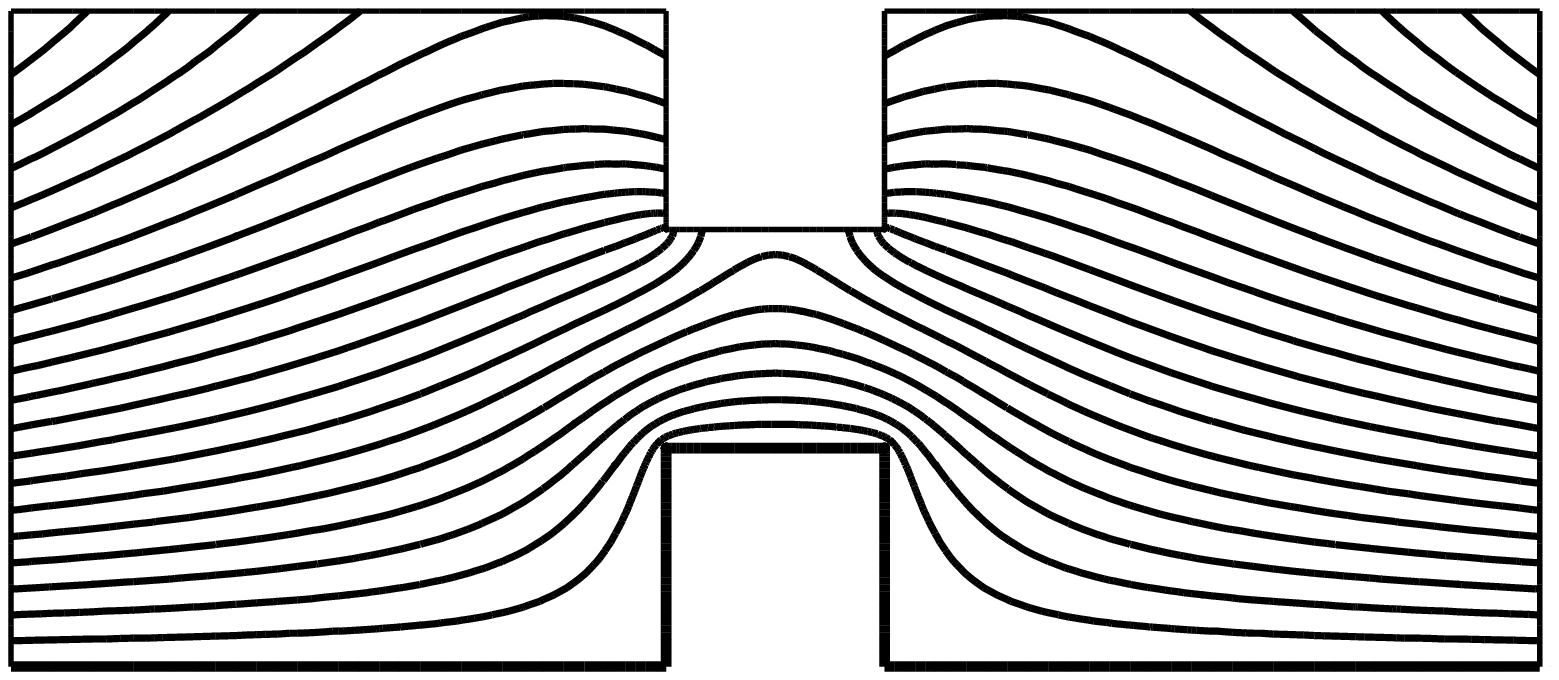} \quad
\includegraphics[width=0.48\textwidth]{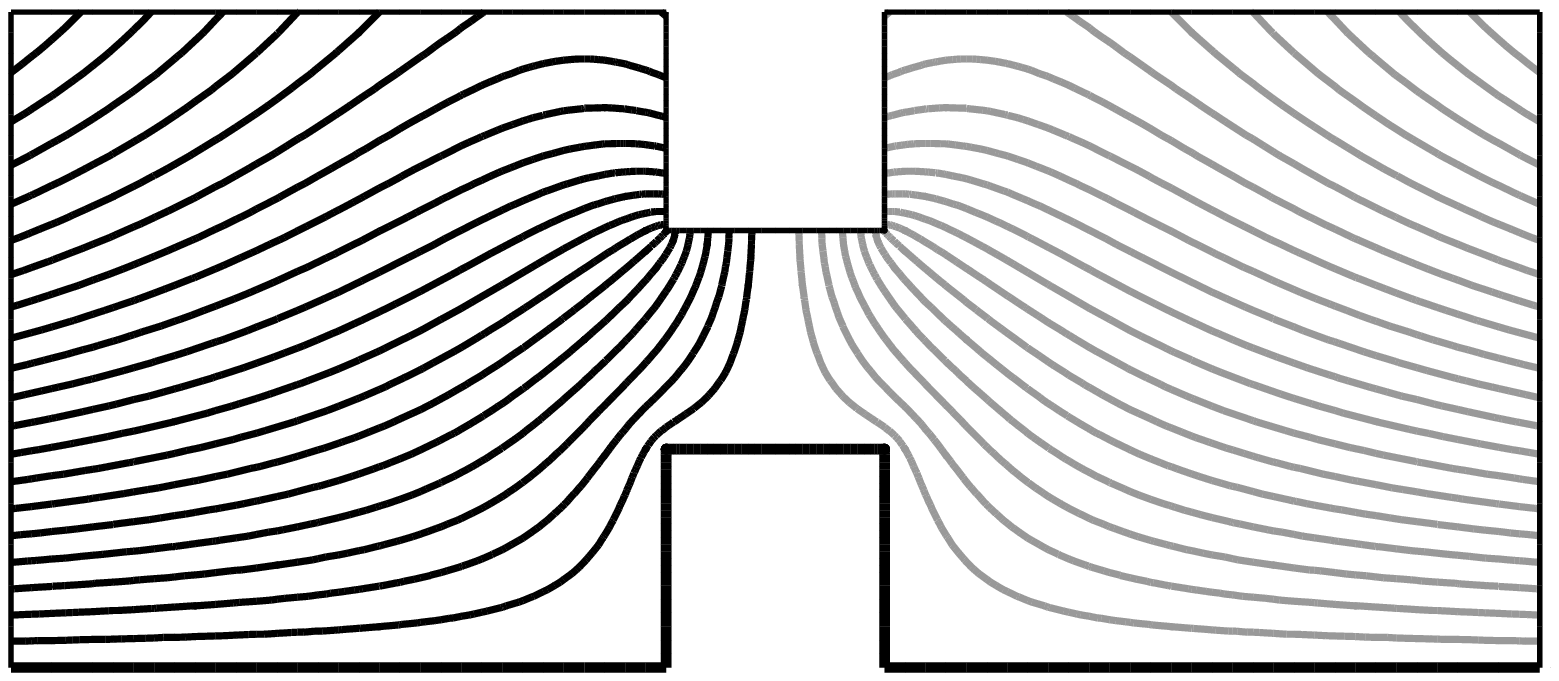}
\caption{\label{fi:eigfun}
Contour plot of eigenfunctions corresponding to the first (left) and to the second (right) eigenvalue.
Eigenfunctions are normalized to have the maximum equal to one. 
The black contour lines correspond to function values $1/20$, $2/20$, \dots, $19/20$ and the grey 
contour lines to $-1/20$, $-2/20$, \dots, $-19/20$.
}
\end{figure}

To present the adaptive process, we plot the evolution of the computed bounds 
for the first four pairs of eigenvalues in Figure~\ref{fi:convergence}. 
As expected, all upper bounds monotonically decrease, because the meshes are nested.
Interestingly, we also observe monotone increase of the lower bounds. 
This is a strong indication that the computed lower bounds are really below 
the true eigenvalues.

This confidence is even higher if the bounds pass the relative closeness test \eqref{eq:rctest}.
Using the lower and upper bound from the final adaptive step, we
perform the relative closeness test \eqref{eq:rctest} for all previous adaptive steps.
In Figure~\ref{fi:convergence}, we show in grey those data points that fail this test.
The points that pass it are plotted in black
and for them we have a good confidence that the relative
closeness assumption is satisfied and that the computed lower bounds 
are really below the exact eigenvalues. 
Notice that this is the case for all eight eigenvalues shown in Figure~\ref{fi:convergence}
except for $\lambda_5$ and $\lambda_7$.

For these two eigenvalues, the distance to the following eigenvalue is too small
to be resolved with the chosen tolerance $\Ereltol = 0.01$.
In the case of $\lambda_5$, the intervals defined by the computed two-sided bounds of $\lambda_5$ and $\lambda_6$ 
overlap and the approximation $\lambda_5^{\cT_k}$ does not pass
the relative closeness test \eqref{eq:rctest} even in the final adaptive step.
The same holds for $\lambda_7$.

This situation, however, does not imply that the relative closeness assumption 
is not satisfied or that the computed lower bound is not below the exact eigenvalue.
In fact, based on an extrapolation of computed eigenvalues, it seems that the 
resulting lower bounds are really below the exact eigenvalues
and that the relative closeness condition \eqref{eq:closestcondi} is satisfied for
the last adaptive steps even for $\lambda_5$ and $\lambda_7$.


\begin{figure}
\includegraphics[width=0.48\textwidth]{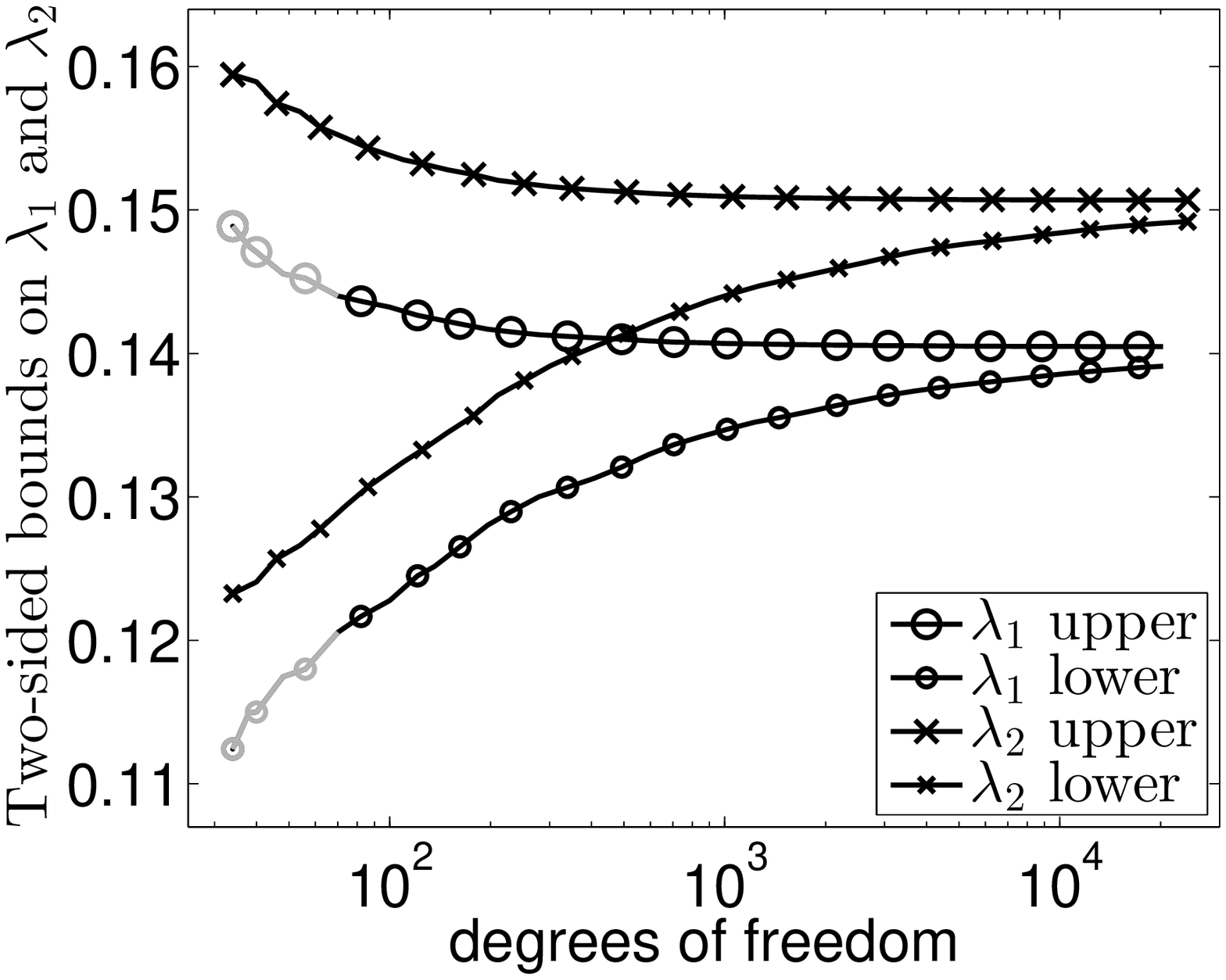}\quad%
\includegraphics[width=0.48\textwidth]{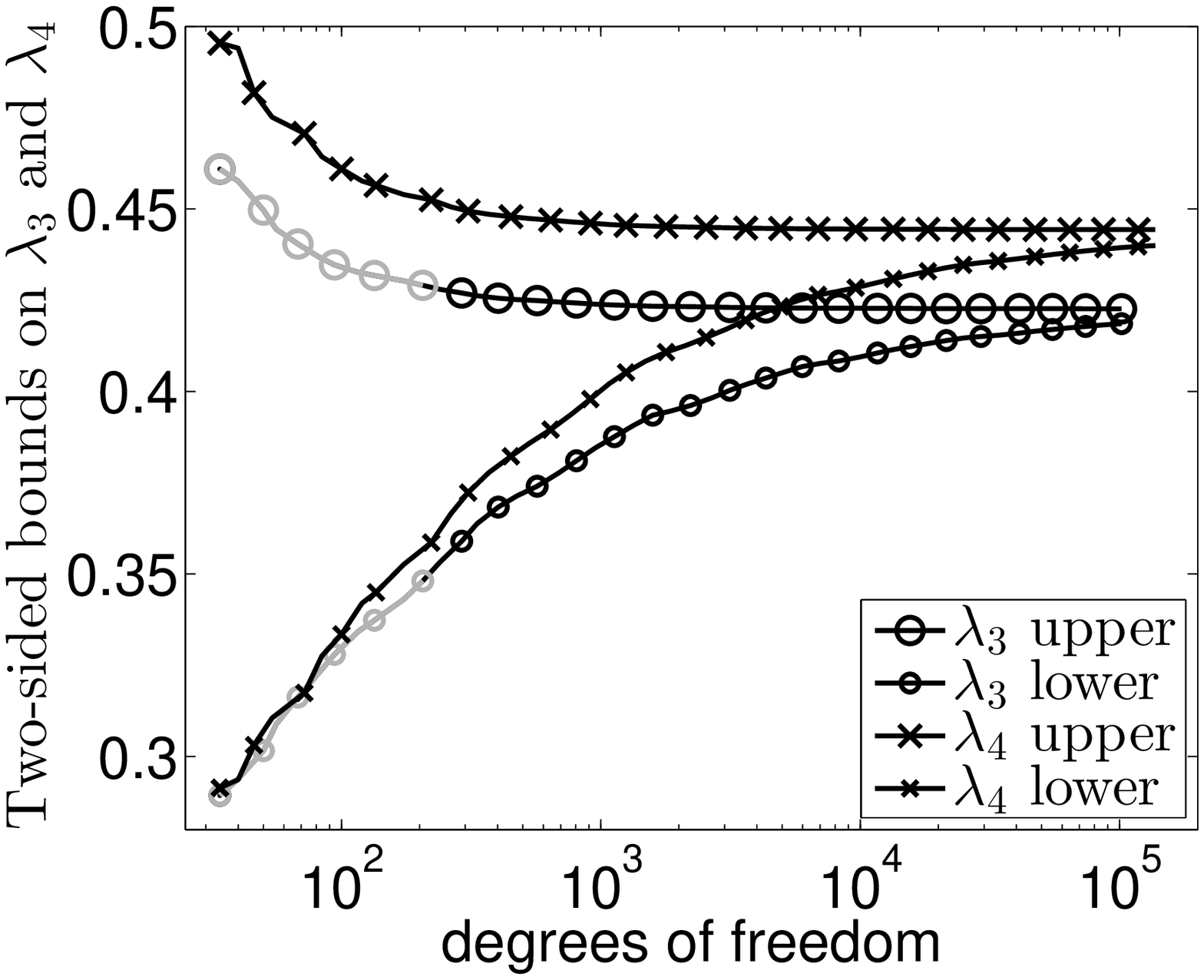}
\\
\includegraphics[width=0.48\textwidth]{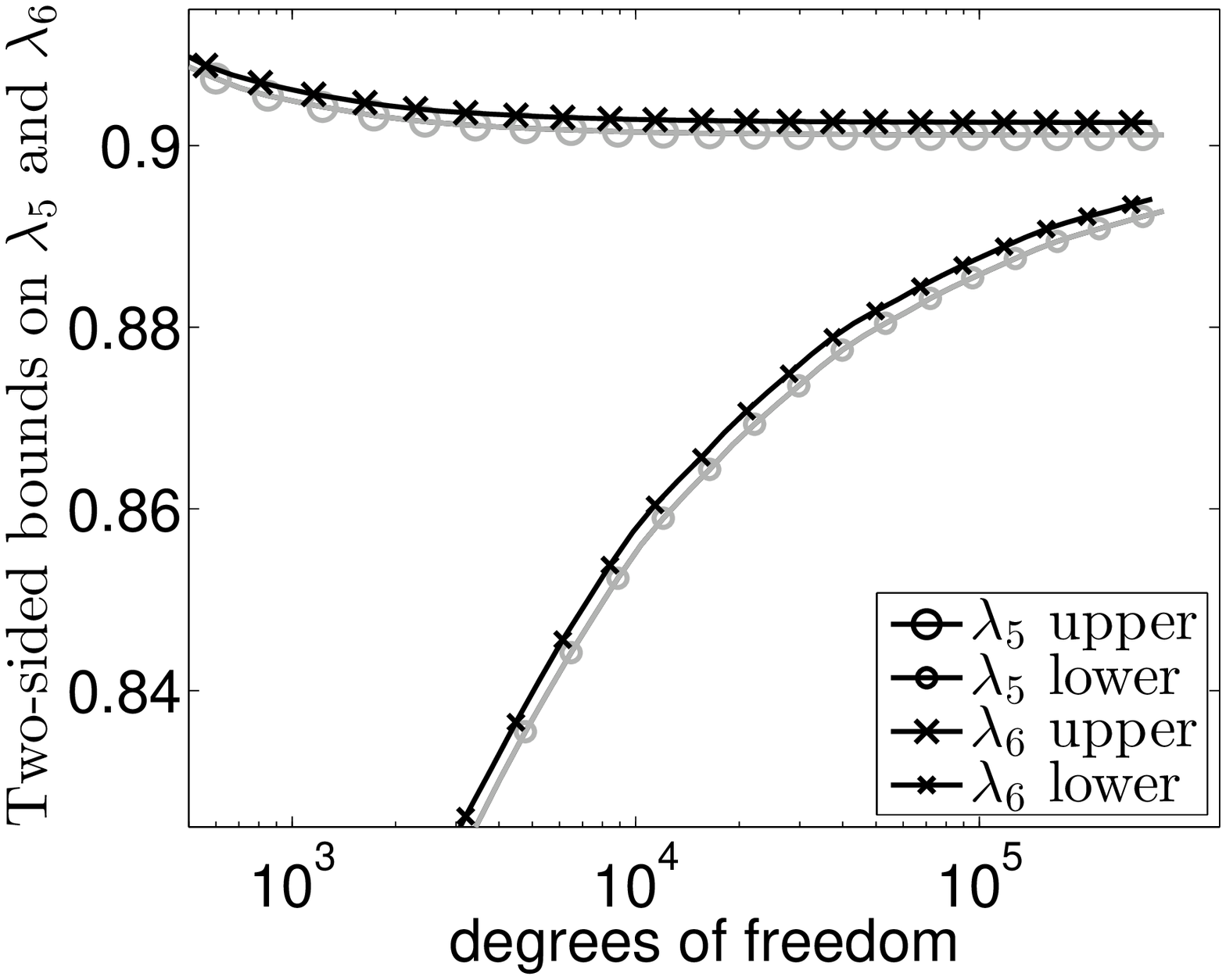}\quad %
\includegraphics[width=0.48\textwidth]{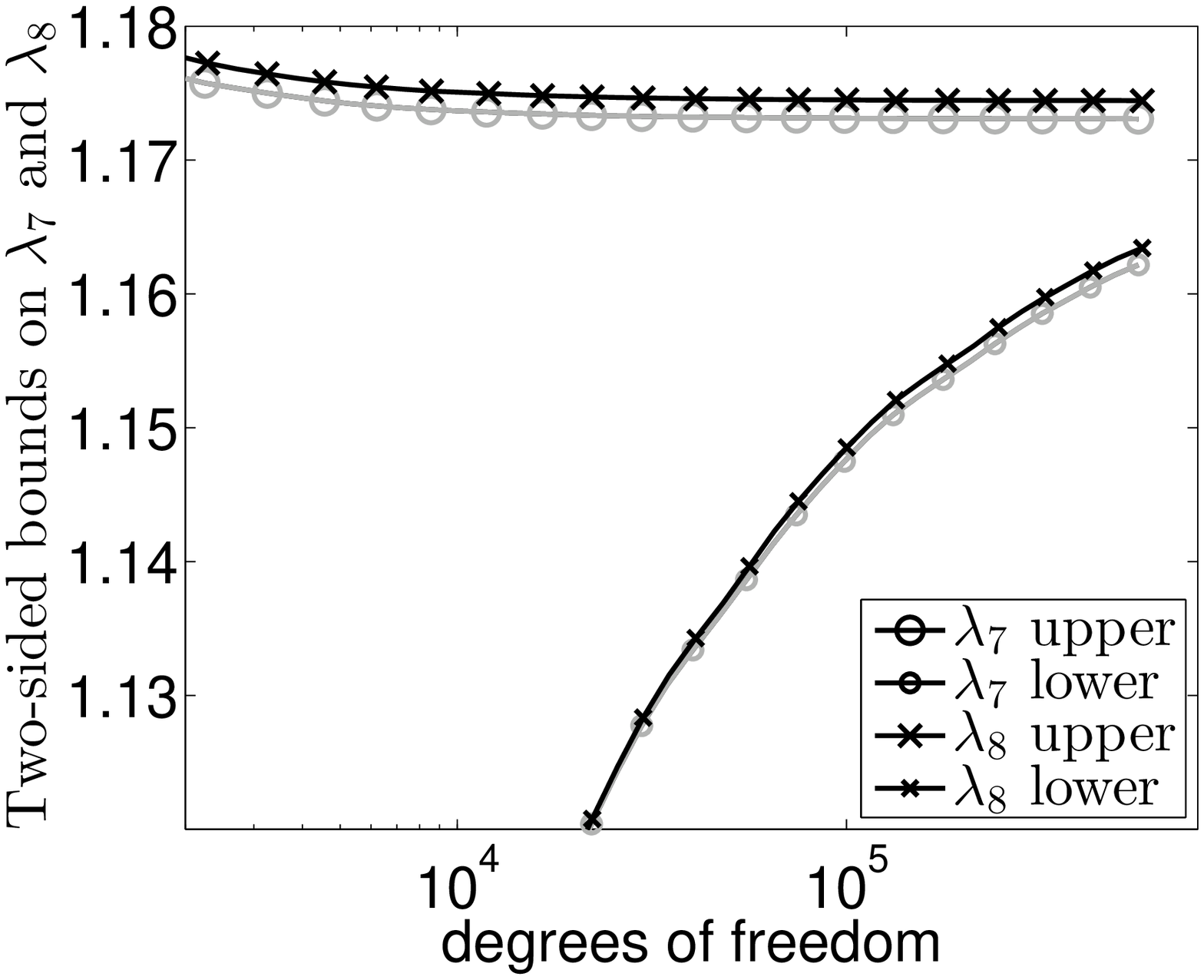}%
\caption{\label{fi:convergence}
Evolution of the lower and upper bounds during the adaptive process.
The bottom panels show only the last adaptive steps in order to visualize small
differences between the bounds. The markers correspond to odd numbers of adaptive steps.
}
\end{figure}

\section{Conclusions}
\label{se:concl}

In this paper, we define a general symmetric elliptic eigenvalue problem \eqref{eq:eigp_strong} and solve it by the standard conforming finite element method to obtain natural upper bounds on the exact eigenvalues. We propose to use complementarity based a posteriori error estimates to compute the corresponding lower bounds. We improve our previous result \cite{Seb_Vejch_2sidedb_eigen_Fr_Poin_trace_14} by reconstructing the flux locally on patches of elements using the reconstruction proposed in \cite{BraSch:2008}. The local flux reconstruction makes the method computationally efficient and since the local problems on patches are independent, they can be solved in parallel.

The main results are theoretical. First, we prove the local efficiency of the proposed error indicator by comparing it to the standard explicit residual error estimator. Second, we prove the convergence of the corresponding adaptive algorithm. To this end we utilize the results in 
\cite{GarMor2011} and verify that the proposed error indicator satisfies the required assumptions. 

The method guarantees lower bounds on the exact eigenvalues only if the relative closeness conditions \eqref{eq:closestcond} and \eqref{eq:closestcondi} are satisfied. These conditions are difficult to guarantee a priori. However, if we compute the two-sided bounds of eigenvalues with sufficient accuracy, we can retrospectively verify the validity of these conditions using \eqref{eq:rctest} and have a good confidence that the computed lower bounds are really below the true eigenvalues. Interestingly, in the performed numerical experiments, the method yielded lower bounds on the true eigenvalues even on rough meshes -- as far as we can judge from the finest two-sided bounds computed. In addition, the computed lower bounds monotonically increase during the adaptive process. All these facts further increase the confidence that the computed lower bounds are really below the exact eigenvalues.

The presented method is quite flexible and enables to compute two-sided bounds of eigenvalues for a wide range of symmetric elliptic eigenvalue problems. Upper bounds are computed by the standard finite element method and lower bounds by flux reconstructions of the finite element eigenfunctions. This reconstruction is efficient, because it is based on solving small problems on patches of elements.
The resulting lower bound is proved to be below the exact eigenvalue if the relative closeness condition holds. This condition cannot be guaranteed, but it can be tested and we can have a good confidence in its validity. We believe that these properties are quite favourable and make this method practical for real applications.




\bibliographystyle{siam}
\bibliography{bibl}

\end{document}